\documentclass[a4paper,leqno]{amsart}

\usepackage[T1]{fontenc}

\usepackage[margin=1.5in]{geometry} 

\usepackage{thmtools}
\usepackage{etoolbox}

\usepackage{enumitem}

\usepackage{amscd}
\usepackage{amsfonts}
\usepackage{amssymb}
\usepackage{amsmath}
\usepackage{amsthm}
\usepackage{array}
\usepackage{breqn}
\usepackage{ifthen}
\usepackage{graphicx}
\usepackage{mathrsfs}
\usepackage{minibox}
\usepackage{multicol}
\usepackage{soul}
\usepackage{stmaryrd}
\SetSymbolFont{stmry}{bold}{U}{stmry}{m}{n}
\usepackage{suffix}
\usepackage{verbatim}
\usepackage{xparse}
\usepackage{xspace}
\usepackage{xstring}

\usepackage[british]{babel}
\usepackage{csquotes}

\usepackage[backend=biber,bibstyle=ext-alphabetic,citestyle=ext-alphabetic,articlein=false,url=false]{biblatex}
\usepackage{hyperref} 
\usepackage{cleveref}

\usepackage{tikz}

\makeatletter
\patchcmd{\thmt@mklistcmd}{\thmt@thmname}{}{}{}
\patchcmd{\thmt@mklistcmd}{\thmt@thmname}{}{}{}
\makeatother

\declaretheoremstyle[
spaceabove=1em,
spacebelow=1em,
headfont=\normalfont\bfseries,
notefont=\mdseries,
notebraces={(}{)},
bodyfont=\normalfont\color{black},
postheadspace=1em,
]{mystyle}

\declaretheoremstyle[
spaceabove=1em,
spacebelow=1em,
headfont=\normalfont\itshape,
notefont=\mdseries,
notebraces={(}{)},
bodyfont=\normalfont\color{black},
postheadspace=1em,
qed=$\blacksquare$
]{myproofstyle}

\declaretheoremstyle[
spaceabove=1em,
spacebelow=1em,
headfont=\normalfont\bfseries,
notefont=\mdseries,
notebraces={(}{)},
bodyfont=\normalfont\color{black},
postheadspace=1em,
qed=$\blacktriangle$
]{myexstyle}

\makeatletter
\patchcmd{\@counteralias}
{\@ifdefinable{c@#1}}
{\expandafter\@ifdefinable\csname c@#1\endcsname}
{}{}
\makeatother

\numberwithin{equation}{section}
\declaretheorem[name=Theorem,sibling=equation,Refname={Theorem,Theorems},style=mystyle]{thm}

\declaretheorem[name=Example,sibling=equation,style=myexstyle]{ex}
\declaretheorem[name=Definition,sibling=equation,style=mystyle]{defn}
\declaretheorem[name=Lemma,Refname={Lemma,Lemmas},sibling=equation,style=mystyle]{lem}
\declaretheorem[name=Corollary,sibling=equation,style=mystyle]{cor}

\declaretheorem[name=Proposition,sibling=equation,style=mystyle]{prop}

\declaretheorem[name=Algorithm,sibling=equation,style=mystyle]{algo}

\crefname{cor}{Corollary}{Corollaries}

\declaretheorem[name=Proof,numbered=no,style=myproofstyle]{proof}

\numberwithin{equation}{section}

\usepackage{JaredLayout}

\newcommand{\charmat}[4]{\trpos*{\begin{pmatrix} #1 & #2\end{pmatrix}} \\ \trpos*{\begin{pmatrix} #3 & #4\end{pmatrix}}}

\let\com\iffalse

\setlist[itemize]{noitemsep, topsep=0pt}
\setlist[enumerate]{noitemsep, topsep=0pt}

\begin{document}


\let\newdommand\NewDocumentCommand
\let\Pr\undefined
\newcommand{\ifmeron}[1]{\IfNoValueTF{#1}{}{#1}}%
\newcommand{\ifmeronpar}[1]{\IfNoValueTF{#1}{}{(#1)}}%
\newcommand{\ifmeronsub}[1]{{\IfNoValueTF{#1}{}{_{#1}}}}%
\newcommand{\ifmeronsup}[1]{\IfNoValueTF{#1}{}{^{#1}}}%
\newcommand{\newtmpvar}[2]{\newcommand{#1}{#2}}%

\NewDocumentEnvironment{tmpvar}{}{\let\tmp\undefined}%

\newdommand{\depcommand}{mO{0}m}{\newcommand{#1}[#2]{{\color{red}#3}}}%
\newdommand{\depdommand}{mmm}{\newdommand{#1}{#2}{{\color{red}#3}}}%


\newcommand{\IZpos}{\IZ_{>0}}%
\newcommand{\IZnn}{\IZ_{\geq 0}}

\newcommand{\primeset}{\ensuremath{S}}
\newcommand{\mS}{\ensuremath{m_\primeset}}
\newcommand{\frakmS}{\ensuremath{\frakm_\primeset}}
\newcommand{\PS}{\ensuremath{P_\primeset}}


\newcommand{\nf}{K}%
\newcommand{\nfup}{L}%
\newcommand{\nfupup}{M}%
\newcommand{\nfext}{\nfup/\nf}%
\newcommand{\nfextbig}{\nfupup/\nf}%
\newcommand{\nfextup}{\nfupup/\nfup}%

\newcommand{\pl}{v}%
\newcommand{\plup}{w}%
\newcommand{\plext}{\plup/\pl}%

\newcommand{\lf}{\nf_\pl}%
\newcommand{\lfup}{\nfup_\plup}%
\newcommand{\lfupup}{\nfupup'}%
\newcommand{\lfext}{\lfup/\lf}%
\newcommand{\lfextup}{\lfupup/\lfup}%
\newcommand{\lfextbig}{\lfupup/\lf}%

\newdommand{\galgp}{O{\nfext}}{ \Gal(#1) }%
\newdommand{\decgp}{}{D(\plext)}%

\newcommand{\galgpHnf}[1]{\pmb{G}_{#1}}%
\newcommand{\galgpHcmfr}[1]{\pmb{G}(\,{#1}\,)}%


\newcommand{\cmf}{K}%
\newcommand{\cmt}{\Phi}%

\newcommand{\galcl}{L}%

\newcommand{\cconj}[1]{\overline{#1}}%
\newcommand{\cconjmap}{\rho}%

\newcommand{\embtoC}{\iota_\IC}%

\newcommand{\tmpcmf}{K'}%

\newcommand{\cmfext}{\cmfup/\cmf}%

\newcommand{\cmfor}{{\cmf_0^r}}%
\newcommand{\cmfo}{\cmf_0}%
\newcommand{\cmfr}{{{\cmf}^r}}%
\newcommand{\cmtr}{\cmt^r}%

\newcommand{\cmfox}{\cmf_0^\times}%

\newcommand{\cmfcmt}{(\cmf, \cmt)}%
\newcommand{\cmfrcmtr}{(\cmfr, \cmtr)}%

\newdommand{\nm}{mo}{N_{#1}\ifmeronpar{#2}}%
\newcommand{\typnm}{\nm{\cmt}}%
\newcommand{\typnmr}{\nm{\cmtr}}%
\newdommand{\relnm}{o}{\nm{\cmf/\cmfo}\ifmeronpar{#1}}%
\newcommand{\absnm}[1]{\nm{#1/\IQ}}%


\newcommand{\cond}[1]{\frakf_{#1}}%
\newcommand{\reld}[1]{\frakd_{#1}}%

\newcommand{\mdls}{\frakm}%
\newdommand{\mdlsfun}{o}{ \mdls\ifmeronpar{#1} }%
\newcommand{\mdlsZ}{m}%

\newdommand{\Id}{moo}{I_{#1}\ifmeronsup{#3}\ifmeronpar{#2}}%
\newdommand{\Pr}{moo}{P_{#1}\ifmeronsup{#3}\ifmeronpar{#2}}%
\newdommand{\clgp}{moo}{\Cl_{#1}\ifmeronsup{#3}\ifmeronpar{#2}}%
\newdommand{\srcg}{mo}{\frakC_{#1}\ifmeronpar{#2}}%
\newcommand{\csbgp}{C}
\newdommand{\Idnf}{o}{ \Id{\nf}[#1] }%
\newdommand{\Idcmf}{o}{ \Id{\cmf}[#1] }%
\newdommand{\Idcmfo}{o}{ \Id{\cmfo}[#1] }%
\newdommand{\Idcmfr}{o}{ \Id{\cmfr}[#1] }%
\newdommand{\Idcmfrcmtr}{o}{ \Id{\cmfr,\cmtr}[#1] }%
\newdommand{\Prnf}{o}{ \Pr{\nf}[#1] }%
\newdommand{\Prcmf}{o}{ \Pr{\cmf}[#1] }%
\newdommand{\Prcmfr}{o}{ \Pr{\cmfr}[#1] }%
\newdommand{\Clnf}{o}{ \clgp{\nf}[#1] }%
\newdommand{\Clcmf}{o}{ \clgp{\cmf}[#1] }%
\newdommand{\Clcmfo}{o}{ \clgp{\cmfo}[#1] }%
\newdommand{\Clcmfr}{o}{ \clgp{\cmfr}[#1] }%
\newdommand{\Clcmfor}{o}{ \clgp{\cmfor}[#1] }%
\newdommand{\srcgcmf}{o}{ \srcg{\cmf}[#1] }%
\newdommand{\srcgcmfr}{o}{ \srcg{\cmfr}[#1] }%

\newdommand{\hcf}{moo}{H_#1\ifmeronsup{#3}\ifmeronpar{#2}}%
\newdommand{\Hnf}{o}{ \hcf{\nf}[#1] }%
\newdommand{\Hcmf}{o}{ \hcf{\cmf}[#1] }%
\newdommand{\Hcmfr}{o}{ \hcf{\cmfr}[#1] }%
\newdommand{\Hcmfo}{o}{ \hcf{\cmfo}[#1] }%
\newdommand{\Hcmfor}{o}{ \hcf{\cmfor}[#1] }%

\newcommand{\nfmone}{\ensuremath{\nf_{\frakm, 1}}}%
\newcommand{\onemodstar}{\ensuremath{1 \mod^\ast}}%
\newcommand{\cmfone}{\cmf_{\mdls, 1}}%


\newdommand{\CMcmf}{o}{ \CM_{\cmf}\ifmeronpar{#1} }%
\newdommand{\CMcmfrcmtr}{o}{ \CM_{\cmfr, \cmtr}\ifmeronpar{#1} }%
\newdommand{\XiKrPhir}{o}{ \Hcmfor(#1)\CMcmfrcmtr[#1] }%


\newcommand{\cgquo}{G}%
\newcommand{\cgext}{E}%
\newcommand{\cgnor}{A}%
\newdommand{\embprob}{o}{ (\nfext, \vareps_{\ifmeron{#1}} ) }%


\newcommand{\roi}[1]{\calO_#1}%
\newcommand{\roix}[1]{\roi{#1}^\times}%
\newcommand{\roip}[1]{\roi{#1}^+}%
\newcommand{\roixp}[1]{{\roi{#1}^{\times +}}}%

\newcommand{\Onf}{\roi{\nf}}%
\newcommand{\Onfup}{\roi{\nfup}}%
\newcommand{\Ocmf}{\roi{\cmf}}%
\newcommand{\Ocmfup}{\roi{\cmfup}}%
\newcommand{\Ocmfo}{\roi{\cmfo}}%
\newcommand{\Ocmfor}{\roi{\cmfor}}%
\newcommand{\Ocmfr}{\roi{\cmfr}}%
\newcommand{\Ocmfx}{\roix{\cmf}}%
\newcommand{\Ocmfxp}{\roixp{\cmf}}%
\newdommand{\Ocmfonex}{o}{\roix{{\cmf, \IfNoValueTF{#1}{\mdls}{#1}, 1}}}%
\newcommand{\Ocmfox}{\roix{\cmfo}}%
\newcommand{\Ocmfop}{\roip{\cmfo}}%
\newcommand{\Ocmfoxp}{\roixp{\cmfo}}%

\newdommand{\idealstar}{o}{\IfNoValueTF{#1}{\left(\Onf/\mdls\right)^\times}{\left(\Onf/{#1}\right)^\times}}


\newcommand{\genord}[2]{\langle #1 \rangle_{#2}}%
\newcommand{\genordadd}[2]{\frac{\IZ}{#2\IZ} \, #1}%
\newcommand{\genvarstd}{\mathbf{e}}
\newcommand{\ordvar}{d}
\newcommand{\gencnt}{n}
\newcommand{\grp}{G}
\newcommand{\subgrp}{H}
\newdommand{\grpdlg}{o}{V_{\IfNoValueTF{#1}{\grp}{#1}}}

\newcommand{\dlogiso}[1]{\gamma_{#1}}

\newcommand{\fuo}{\vareps_0}%


\newcommand{\effzero}{\pi_m}%
\newdommand{\effone}{o}{\eta}
\newdommand{\efftwo}{o}{r}%

\newcommand{\ifmeronbrak}[1]{\IfNoValueTF{#1}{(m)}{(#1)}}%
\newcommand{\ifmeronnobrak}[1]{\IfNoValueTF{#1}{m}{#1}}%
\newcommand{\effzerotwo}{\pi_2}%
\newcommand{\effonetwo}{\effone[2]}
\newcommand{\efftwotwo}{\effone[2]}

\newcommand{\None}{N_1}%
\newcommand{\Ntwo}{N_2}%
\newcommand{\ibari}[1]{#1\bar{#1}}%

\newcommand{\calAC}{\calA\hspace{-0.2em}\calC}
\newcommand{\srcgelt}{b}

\let\trpos\undefined
\newdommand{\trpos}{sm}{\IfBooleanTF{#1}{#2^\intercal}{\left(#2\right)^\intercal}}

\makeatletter
\newcommand{\labeleqstar}{\Hy@raisedlink{\hypertarget{label:eqstar}{}}}
\makeatother

\newcommand{\eqstar}[1]{(\hyperlink{label:eqstar}{\ensuremath{{\star}_{#1}}})}

\newcommand{\ppas}{A}
\newcommand{\torb}{\calB}
\newdommand{\ppastorb}{o}{(\ppas\ifmeronsub{#1}, \torb\ifmeronsub{#1})}
\newcommand{\ppastorbpr}{(\ppas', \torb')}
\newcommand{\hypell}{C}
\newcommand{\hypinv}{\iota}
\newdommand{\Jac}{o}{J(#1)}

\newcommand{\blambda}{\pmb{\lambda}}
\newcommand{\fromoneto}[1]{\in \{1, \ldots, #1\}}

\newcommand{\bbeta}{\pmb{\beta}}

\newcommand{\galgpblw}{\bbGa}%
\newcommand{\alphar}{\alpha_r}


\newcommand{\torsgen}{t}
\newcommand{\torsord}{T}

\newdommand{\cmfdab}{momm}{\IQ[#1]~/~(#1^4~+~#3~#1^2~+~#4)}
\newdommand{\cmfodab}{momm}{\IQ[#1]~/~(#1^2~+~#3~#1~+~#4)}

\newcommand{\writefilename}{}


\title[Computing the Hilbert Class Fields of Quartic CM Fields Using CM]{Computing the Hilbert Class Fields of Quartic CM Fields Using Complex Multiplication}
\author{Jared Asuncion}
\thanks{INRIA, LFANT, F-33400 Talence, France\\
CNRS, IMB, UMR 5251, F-33400 Talence, France\\
Univ. Bordeaux, IMB, UMR 5251, F-33400 Talence, France}
\thanks{Universiteit Leiden, Postbus 9512, 2300 RA Leiden, The Netherlands}
\thanks{This project has been supported by a travel grant managed by the Agence Nationale de la Recherche under the program ``Investissements d'avenir'' and bearing the reference ANR-10-IDEX-03-02.}

\begin{abstract}
Let $K$ be a quartic CM field, that is, a totally imaginary quadratic extension of a real quadratic number field. In a 1962 article titled \textit{On the class-fields obtained by complex multiplication of abelian
varieties}, Shimura considered a particular family $\{F_K(m) : m \in \IZ_{>0}\}$ of abelian extensions of $K$, and showed that the Hilbert class field $H_K$ of $K$ is contained in $F_K(m)$ for some positive integer~$m$. We make this $m$ explicit. We then give an algorithm that computes a set of defining polynomials for the Hilbert class field using the field~$F_K(m)$. Our proof-of-concept implementation of this algorithm computes a set of defining polynomials much faster than current implementations of the generic Kummer algorithm for certain examples of quartic CM fields.
\end{abstract}

\maketitle

\section{Introduction}
\label{sec:intro}
\writefilename

An abelian extension of a number field $\nf$ is a Galois extension $\nfext$ whose Galois group $\galgp$ is abelian.
The Kronecker-Weber Theorem states that every abelian extension $\nfup$ of $\IQ$ is contained in $\IQ(\exp(2\pi i \tau) : \tau \in \IQ)$,
the field obtained by adjoining to $\IQ$ the values of the analytic map $\tau \mapsto \exp(2\pi i \tau)$ evaluated at $\IQ$. The twelfth out of the twenty-three problems Hilbert posed in 1900 asks, roughly speaking, whether a statement analogous to the Kronecker-Weber Theorem can be made for number fields $K$ different from $\IQ$.

One tool in attacking this problem is the theory of complex multiplication (CM),
developed by Shimura \cite{shimura1971,shimura2016}
during the second half of the 20th century.
CM theory works for a specific family of number fields, called CM fields.
A CM field $\cmf$ is a totally imaginary quadratic extension of a totally real number field $\cmfo$.

For a CM field $\cmf$,
class field theory proves the existence of the family
$$\calF~=~\{\Hcmf[m]~:~m~\in~\IZ\}$$
of \textit{ray class fields} of $\cmf$.
Every finite degree abelian extension of $\cmf$
is contained in at least one element of $\calF$. Given a CM field $\cmf$ and a CM type $\cmt$ (defined in \Cref{subsec:cm}), CM theory gives a reflex pair $\cmfrcmtr$ consisting of a CM field $\cmfr$ and CM type $\cmtr$, and defines a certain field $\CMcmfrcmtr[m]$ such that
\begin{itemize}
	\item $\CMcmfrcmtr[m] \subseteq \Hcmfr[m]$, and
	\item $\CMcmfrcmtr[m]$ is an algebraic extension of $\cmfr$ that can be obtained by adjoining to $\cmfr$ algebraic numbers which are \textit{special values of modular functions}. These are the analogue of $\exp(2\pi i \tau)$.
\end{itemize}

For the case when $\cmf$ is a CM field of degree $2$, it has been shown that $\cmf = \cmfr$ and $\CMcmfrcmtr[m]$ is exactly the ray class field $\Hcmf[m]$. This is not the case for higher-degree CM fields. In fact, the field $\cmf$ is not necessarily equal to $\cmfr$ for the next simplest case -- CM fields of degree $4$. One natural question to then ask is: can one use CM theory to compute ray class fields $\Hcmf[m]$ of quartic CM fields $\cmf$, and if so how?

We answer this question for the case of $\Hcmf[1]$. We show how one can use CM theory to compute the Hilbert class field $\Hcmf[1]$, the largest unramified abelian extension, of a quartic CM field $\cmf$. Using this theoretical result, we give an algorithm which computes $\Hcmf[1]$ for quartic CM fields satisfying properties elaborated on later in this article. Finally, we write a proof-of-concept implementation of this algorithm, which we use to find defining polynomials of Hilbert class fields of quartic CM fields. Using this implementation, we found that this algorithm succeeds to compute the Hilbert class field of certain quartic CM fields in which current implementations of the well-known Kummer theory algorithm take much longer.

\begin{tmpvar} \newcommand{\tif}{K}%
For any number field $\tif$ and any positive integer $m$, let
$E_\tif(m)$ be the smallest subfield of $H_\tif(m)$ containing $\tif$
such that $\Gal(H_\tif(m) / E_\tif(m))$ is of exponent at most $2$.
The main theoretical result of this article, which we specialize to the case of primitive quartic CM fields in \cref{cor:main1}, is as follows.
\end{tmpvar}

\begin{thm} \label{thm:main1}
Let $K$ be a number field without real embeddings. Let $\primeset$ be a finite set of prime ideals of $\Onf$ such that
\begin{itemize}
\item $| \Cl_K(1) / \langle \primeset \rangle |$ is odd,
\item $\primeset$ contains all prime ideals above $2$,
\item $\primeset$ contains at least $3$ elements.
\end{itemize}
Let $\PS = \{ p : p \text{ is a rational prime below } \frakp \text{ for some } \frakp \in \primeset \}$.
Let $\mS = 4 \cdot \prod_{p \in P_S} p$. Then $H_K(1) \subseteq E_K(\mS)$.
\end{thm}

The existence of a positive integer $m$ such that $\Hcmf[1] \subseteq E_K(m)$ was already known by Shimura, via the proof of \cite[Theorem 2]{shimura1962}. \Cref{thm:main1} gives a formula for such an $m$ thereby making an effective version of Shimura's result.

We will be using a result of Shimura \cite{shimura1962}, later refined by Streng, which gives an example of a field that is `close' but not quite the Hilbert class field.

\begin{thm}[label={thm:shimuracontainmentbelow},name={\cite[Theorem I.10.3]{streng2010}}]
	Let $\cmf$ be a quartic CM field which is not bicyclic Galois\footnote{Equivalently, it is either cyclic Galois or not Galois.} and $\cmt$ a CM type of $\cmf$, and let $\cmfrcmtr$ be the reflex CM pair of $\cmfcmt$.
	Then the Galois group $\Gal(\Hcmfr(m) / \XiKrPhir[m])$ is abelian of exponent at most $2$ for any positive integer $m$.
\end{thm}

There is also analogous result \cite[Theorem I.10.5]{streng2010} similar to the above theorem when $K$ is bicyclic Galois.

\Cref{thm:shimuracontainmentbelow} tells us that the field $\Hcmfr[1]$ is obtained by adjoining square roots of elements from the compositum $\XiKrPhir[1]$.
One can determine which square roots must be added using a generic algorithm
given by Kummer theory.
However, the Kummer theory algorithm requires
the computation of ray class groups of the compositum $\XiKrPhir[1]$.
Algorithms to compute class groups are known to not perform very well \cite[page 3]{belabas2004} for large-degree number fields.
We avoid this issue by instead computing a larger field containing the Hilbert class field $\Hcmfr[1]$ and using Galois theory to eventually find a defining polynomial for $\Hcmfr[1]$. Our method only involves computing ray class groups of number fields of degree at most $4$.

Suppose that $m$ is an integer and that $\cmfcmt$ and $\cmfrcmtr$ are as in the assumptions of \Cref{thm:shimuracontainmentbelow}. We denote by \eqstar{m} the expression
\begin{equation} \label{eq:star} \labeleqstar %
\Hcmfr(1) \subseteq \XiKrPhir[m] \tag{\ensuremath{\star_m}}.
\end{equation}

\begin{cor}[label={cor:main1}]
\newcommand{\qcmf}{F}
	Let $\cmfcmt$ be a CM pair, with $\cmf$ being a quartic CM field which is not bicyclic Galois, and let $\cmfrcmtr$ be its reflex CM pair. Let $\mS$ be as in \Cref{thm:main1}.
	Then \eqstar{{\mS}} holds.
\end{cor}

\begin{proof}
\Cref{thm:shimuracontainmentbelow} gives $$H_K(1) \subseteq E_K(m_S)$$ and \Cref{thm:main1} gives $$E_K(m_S) \subseteq \XiKrPhir[m].$$
\end{proof}

Using this corollary, one can compute 
$$
\XiKrPhir[\mS]
$$
then use Galois theory to finally compute a defining polynomial for the subfield we are interested in, the Hilbert class field $\Hcmfr[1]$.

This article is divided into several sections. In \Cref{sec:prelims}, we define objects and recall results from class field theory and CM theory that we use for the rest of the article. In \Cref{sec:mainthm}, we prove \Cref{thm:main1}. In \Cref{sec:decisionprob}, we give an algorithm that tells us whether or not $\eqstar{m}$ holds for any integer $m$. %

In \Cref{subsec:algoimpl}, we discuss how to obtain, in the case where $\eqstar{2}$ holds, a set $\pmb{\beta}$ of algebraic numbers such that $\Hcmfr[1] = \cmfr(\pmb{\beta})$. In that section, we also discuss how our implementation of this algorithm fares against the generic Kummer theory algorithm.

One may note that the formula for $\mS$ in \Cref{cor:main1} implies that $\mS \geq 8$. An algorithm to compute a defining polynomial for $\CMcmfrcmtr[m] / \cmfr$ where $m > 2$ will be the topic of a future article.

\section{Preliminaries}
\label{sec:prelims}
\writefilename

In this section, we define the mathematical objects informally introduced in \Cref{sec:intro}.

\Cref{subsec:cft} is dedicated to reviewing class field theory and the ray class field $\Hnf(m)$ of a number field $\nf$.

\Cref{subsec:cm} reviews the theory of complex multiplication, CM theory. There, we define CM fields, their reflexes, type norms and the field $\CMcmfrcmtr[m]$ referenced in this article's introduction.

  \subsection{Class Field Theory}
  \label{subsec:cft}
\writefilename

This section is concerned with defining class field theory concepts needed for this article. For a more thorough treatment of class field theory, see~\cite{neukirch1999}.

\renewcommand{\nf}{F}
Let $\nf$ be a number field and denote by $\Onf$ its ring of integers. 

A \textit{fractional ideal} of $\Onf$ is an $\Onf$-submodule
$\frakn$ of $\nf$
such that there exists a $d \in \IZ$ such that $d\frakn$ is a nonzero ideal of $\Onf$.
Denote by $\Idnf$ the group of fractional ideals of $\Onf$
and denote by $\Prnf$ its subgroup of fractional ideals generated
by a single element.
The quotient $\Idnf / \Prnf$ forms a finite group which we call
\textit{the ideal class group} of $\nf$, denoted by $\Clnf$.

\begin{tmpvar} \newcommand{\PrIdSet}{\calP}%
Let $\PrIdSet$ be the union of the set of prime ideals of $\Onf$ (finite places) and the set of real embeddings and conjugate pairs of complex embeddings of $\nf$ (infinite places).
A \textit{modulus} $\mdls$ for $\nf$ is a function
$\mdlsfun: \PrIdSet \to \IZnn$
such that $\mdlsfun[\frakp] = 0$
for all but finitely many prime ideals $\frakp$, and
$\mdlsfun[\sigma] \leq 1$ when $\sigma$ is a real embedding of $\nf$,
and $\mdlsfun[\sigma] = 0$ when $\sigma$ is a complex embedding
of $\nf$. If $\nf$ is totally imaginary,
meaning that is has no real embeddings,
then the map
$$
\fraka
\hspace{1em}
\mapsto 
\hspace{1em}
\left(
\hspace{0.5em}
{\renewcommand{\mdls}{\frakm_\fraka}
\mdlsfun}
:\,
\frakp \mapsto \begin{cases}
\ord_\frakp(\fraka) & \frakp \text{ is a finite place} \\
0 & \frakp \text{ is an infinite place} \\
\end{cases}
\hspace{0.5em}
\right)
$$
is a bijection
from the set of nonzero ideals of $\Onf$
to the set of moduli on $\nf$.
In such a case, we may interchangeably use
the terms `modulus' and `ideal of $\Onf$'.

For a modulus $\mdls$,
we denote by $\Idnf[\mdls]$ the subgroup of $\Idnf$
composed of the fractional prime ideals $\frakp$ of $\Onf$
which satisfy $\mdlsfun[\frakp] = 0$.
We define $\nfmone$ to be the set of $a \in \nf^\times$
such that
$\ord_\frakp(a - 1) \geq \mdlsfun[\frakp]$
for all finite primes $\frakp$ with $\mdlsfun[\frakp] \geq 0$
and $\sigma(a) > 0$ for any $\sigma$ such that $\mdlsfun[\sigma] = 1$.
The statement `$a \in \nfmone$'
is more commonly denoted as $a \equiv \onemodstar \mdls$.
The latter notation takes preference in this article.
Denote by $\Prnf[\mdls]$ the subgroup of $\Idnf[\mdls]$
generated by fractional ideals of $\Onf$
generated by elements of $\nfmone$.
The quotient $\Idnf[\mdls] / \Prnf[\mdls]$ is a finite group which we call
\textit{the ray class group} of $\nf$ for the modulus $\mdls$, denoted by $\Clnf[\mdls]$.

One of the main results of class field theory
is: given a modulus $\mdls$,
there exists a finite abelian extension $\Hnf[\mdls] / \nf$, called the \textit{ray class field} of $\nf$
for the modulus $\mdls$,
which satisfies
\begin{itemize}
\item $\Hnf[\mdls]$ is unramified at all primes $v \in \PrIdSet$ with $\mdlsfun[v] > 0$,
\item $\Clnf[\mdls]$ and $\Gal(\Hnf[\mdls]/\nf)$ are isomorphic via the Artin map \cite[Theorem VI.5.5]{neukirch1999}.
\end{itemize}
\end{tmpvar}

A subgroup $\csbgp$ of $\Idnf[\mdls]$ such that $\Prnf[\mdls] \subseteq H$
is called \textit{a congruence subgroup} modulo $\frakm$.
Let $\frakm$ be a modulus for $\nf$.
Galois theory gives a bijection between the set of
congruence subgroups
modulo $\frakm$
and the abelian extensions $\nfext$ of $\nf$ such that $\nfup \subseteq \Hnf[\mdls]$.
In particular, if $\csbgp$ is a congruence subgroup modulo
$\frakm$ which corresponds to the abelian extension $\nfup$ of $\nf$, 
contained in $\Hnf[\mdls]$,
then
$$
\Gal(\Hnf[\mdls] / \nfup) \, \cong \, \csbgp/\Prnf[\mdls]
\qquad
\text{and}
\qquad
\Gal( \nfup / \nf ) \, \cong \, \Idnf[\mdls]/\csbgp.
$$

\begin{tmpvar} \renewcommand{\mdls}{\frakn}
If $n \in \IZpos$,
and $\mdls(\frakp) = \ord_\frakp(n)$ for all prime ideals of $\Onf$
and $\mdls(\sigma) = 0$ for all real embeddings of $\nf$,
then we may write $\Idnf[n], \Prnf[n], \Clnf[n]$ and $\Hnf[n]$
instead of $\Idnf[\mdls], \Prnf[\mdls]$, $\Clnf[\mdls]$ and $\Hnf[\mdls]$.
With this notation,
$\Hnf = \Hnf[1]$ is \textit{the Hilbert class field} of $\nf$. 
On the other hand,
if $\mdls(\frakp) = \ord_\frakp(n)$ for all prime ideals of $\Onf$
and $\mdls(\sigma) = 1$ for all real embeddings of $\Onf$,
then we may write $\Idnf[n][+], \Prnf[n][+], \Clnf[n][+]$ and $\Hnf[n][+]$
instead of $\Idnf[\mdls], \Prnf[\mdls]$, $\Clnf[\mdls]$ and $\Hnf[\mdls]$.
With this notation,
$\Hnf^+ = \Hnf[1][+]$ is \textit{the narrow Hilbert class field} of $\nf$.
\end{tmpvar}

  \subsection{CM Theory}
  \label{subsec:cm}
\writefilename

In this section, we set notation and
recall results of complex multiplication (CM) theory
that we use to define $\CMcmfrcmtr[m]$
and other objects introduced in later chapters.
For a more complete treatment of CM theory,
the reader is referred to \cite{shimura2016,streng2010}.

\renewcommand{\cmf}{K}%
A \textit{CM field} $\cmf$ is a totally imaginary number field
which is a quadratic extension of a totally real field $\cmfo$.
The degree of a CM field is its degree as a number field.
Hence, any CM field has even degree.
For the rest of \Cref{subsec:cm}, we fix a CM field $\cmf$ of degree $2g$.

Denote by $\cconjmap$ the sole generator
of the group $\Gal(\cmf/\cmfo)$, a group of order $2$.
This can be thought of as the \textit{complex conjugation} morphism
because for any embedding $\phi: \cmf \hookrightarrow \IC$,
we have $\phi(\cconjmap(x)) = \cconj{\phi(x)}$.

Let $\galcl$ be the Galois closure of the CM field $\cmf$,
and fix an embedding $\embtoC: \galcl \to \IC$.
The CM field $\cmf$ has $g$ complex conjugate pairs of embeddings into $\IC$.
By applying the `inverse' of $\embtoC$ to each embedding $\sigma: \cmf \to \IC$, we can think of these $2g$
complex embeddings as embeddings into $\galcl$.
A set $\Phi$ is called a \textit{CM type} of $\nf$ if it contains $g$ complex embeddings of $\nf$ into $\nfup$ such that for any $\phi, \phi' \in \Phi$, $\phi' \neq \phi \circ \cconjmap$.

Let $\cmf_2/\cmf_1$ be an extension of CM fields,
with $\galcl_2$ the Galois closure of $\cmf_2$.
Let $\Phi_1$ be a CM type on $\cmf_1$.
The CM type \textit{induced by $\Phi_1$ on $\cmf_2$} is defined to be
$$
\Phi_2= \{ \phi: \cmf_2 \hookrightarrow \galcl_2 : \phi|_{\cmf_1} \in \Phi_1 \}.
$$
A CM type of $\cmf$ is said to be \textit{primitive} if it is not induced
from a CM type of a strict CM subfield of $\cmf$.
Two CM types of $\Phi, \Phi'$ are \textit{equivalent} if there is an automorphism $\sigma$ of $\cmf$ such that $\Phi' = \Phi\sigma$ holds.

We call a pair $\cmfcmt$ \textit{a CM pair}.
If $\Phi$ is a primitive CM type on $K$, we say that a CM pair $\cmfcmt$ is \textit{primitive}.
Denote by $\Phi_\galcl$ the CM type induced by $\Phi$ on the Galois closure $\galcl$,
which is also a CM field.
Since the elements of $\Phi_\galcl$ are automorphisms of $L$,
we can define the set $\Phi_\galcl^{-1} = \{ \phi^{-1} : \phi \in \Phi_\galcl \}$.
This set $\Phi_\galcl^{-1}$ is a CM type of $\galcl$.
There exists a unique subfield $\cmfr$ of $\galcl$
and a unique CM type $\Phi^r$ on $\cmfr$
such that $\Phi^r$ is a primitive CM-type
which induces $\Phi_\galcl^{-1}$.
The CM pair $\cmfrcmtr$ is called the \textit{reflex} of $\cmt$.
One property of the reflex field $\cmfr$ is that
\begin{equation} \label{eq:reflexfixedfield}
\Gal(\galcl/\cmfr) = \{ \sigma \in \Gal(\galcl/\IQ) : \sigma\Phi = \Phi \}.
\end{equation}
If $\Phi$ is a primitive CM type of $\cmf$,
then the reflex $({\cmfr}^r, {\Phi^r}^r)$ of $\cmtr$ is actually equal to $\cmfcmt$.

Let $\cmfcmt$ be a quartic CM pair and let $\galcl$ be its Galois closure.
There are three possibilities for the Galois group $G = \Gal(\galcl/\IQ)$: $G \cong C_2 \times C_2$, $G \cong C_4$ or $G \cong D_4$.
For the second and third possibilities, $\cmt$ is a primitive CM type,
regardless of the choice of $\Phi$.

Let $\cmfcmt$ be a primitive CM pair
and let $\cmfrcmtr$ be its reflex.
Then we have a map $\typnmr: \cmfr \to \galcl$
defined by $y \mapsto \prod_{\phi^r \in \Phi^r} \phi^r(y)$.
By \Cref{eq:reflexfixedfield},
this is in fact a map to the reflex $\cmfr^r = \cmf$ of $\cmfr$.
We refer to this map as the (reflex) \textit{type norm} map.
Some articles refer to it as a \textit{half norm} map
since it uses half the number of embeddings as the usual norm map.

Similar to the usual norm map,
the type norm map induces maps
$$\typnmr: \Idcmfr[m] \to \Idcmf[m] \quad\,\quad \text{and} \quad\,\quad \typnmr: \Clcmfr[m] \to \Clcmf[m]$$
for any positive integer $m$
using \cite[Lemma I.8.3]{streng2010}, which uses \cite[Remark on page 63]{lang1983} and \cite[Proposition 29]{shimura1998}
in its proof.
The notation $\typnmr$ will be used to denote
any of the above three maps and the domain
will be specified whenever the context of the discussion
does not make it clear.

Keeping the notation from the previous paragraph,
let $\mdls$ be an ideal of $\Ocmf$ and
denote by $m$ the smallest positive integer in $\mdls$.
Define the subgroup $\Idcmfrcmtr[\mdls]$ of $\Idcmfr[m]$ to be
\[
\Idcmfrcmtr[\mdls]
=
\left\{
\fraka \in \Idcmfr[m]
:
\begin{array}{rcl}
\exists x \in {\cmf}^\times \text{ such that}&& \\
N_{\Phi^r}(\fraka) &=& x\Ocmf \\
N_{K^r/\IQ}(\fraka) &=& x\cconj{x} \\
x &\equiv& \onemodstar \mdls\\
\end{array}
\right\}.
\]

Noticing that $N_{\cmfr/\IQ}(x) = \typnmr(x)\cconj{\typnmr(x)}$
for every $x \in \cmfr$, we find that $\Idcmfrcmtr[\mdls]$ is a congruence subgroup modulo $m$.
As a congruence subgroup,
this corresponds to a field extension $\CMcmfrcmtr[\mdls]$ of $\cmfr$
contained in $\Hcmfr[m]$.



\section{An integer \texorpdfstring{$m$}{m} for which \texorpdfstring{$\ref{eq:star}$}{(*m)} holds.}
\label{sec:mainthm}
\writefilename


The aim of this section is to prove \Cref{thm:main1}, which gives a formula to find an integer $m$ such that \eqstar{m} holds. In \Cref{subsec:embprob}, we discuss embedding problems, which we use in \Cref{subsec:embprobbounds} to prove \Cref{thm:main1}.


  \subsection{Embedding problems.}
  \label{subsec:embprob}
\writefilename

We state a result of Richter used in Shimura's original proof \cite[Proof of Theorem 2]{shimura1962}. 

\begin{lem} \label{lem:globemb}
Let $a$ be a non-negative integer. Let $\nf$ be a totally imaginary number field. Let $\nfext$ be an unramified cyclic Galois extension of degree $2^a$. Then there exists a cyclic Galois extension $\nfextbig$ of degree $2^{a+1}$ which contains $\nfup$.
\end{lem}

Even though \Cref{lem:globemb} is a special case of \cite[Satz 1b]{richter1936na}, we will prove it to keep this article self-contained. The proof concerns \textit{embedding problems}, which we define in this section.

Let $\cgquo$ and $\cgnor$ be groups.
A \textit{central group extension of $\cgquo$ by $\cgnor$} is an exact sequence
\begin{equation} \label{eq:nhq}
1 \to \cgnor \too{\iota} \cgext \too{\pi} \cgquo \to 1
\end{equation}
such that
$\iota(\cgnor)$ is in the center of $\cgext$.

\begin{defn}
By an \textit{embedding problem}, we will mean a pair $\embprob$ where $\nfext$ is a Galois extension and $\vareps$ is a central group extension given by an exact sequence
$$
1 \to \cgnor \too{\iota} \cgext \too{\pi} \cgquo \to 1
$$
where $G = \galgp[\nfext]$. A \textit{solution} to such an embedding problem is a Galois extension $\nfextbig$ containing $\nfup$ such that there exists an isomorphism $\phi: \galgp[\nfextbig] \to \cgext$ which induces a commutative diagram
\begin{center}
\begin{tikzpicture}[every node/.style={midway}, baseline=(current  bounding  box.center)]
\matrix[column sep={3em},
row sep={3em}] at (0,0)
{  ;
\node(A1) {$1$}; &
\node(A2) {$\galgp[\nfextup]$}; & 
\node(A3) {$\galgp[\nfextbig]$}; &
\node(A4) {$G$}; & 
\node(A5) {$1$}; \\
\node(B1) {$1$}; &
\node(B2) {$\cgnor$}; & 
\node(B3) {$\cgext$}; &
\node(B4) {$\cgquo$}; & 
\node(B5) {$1$.}; \\
};
\draw[->] (A1) -- (A2) node {};
\draw[->] (A2) -- (A3) node {};
\draw[->] (A3) -- (A4) node {};
\draw[->] (A4) -- (A5) node {};
\draw[->] (B1) -- (B2) node {};
\draw[->] (B2) -- (B3) node[anchor=south] {$\iota$};
\draw[->] (B3) -- (B4) node[anchor=south] {$\pi$};
\draw[->] (B4) -- (B5) node {};
\draw[->] (A2) -- (B2) node[anchor=east] {};
\draw[->] (A3) -- (B3) node[anchor=east] {$\phi$};
\draw[->] (A4) -- (B4) node[anchor=east] {$\id_G$};
\end{tikzpicture}
\end{center}
\end{defn}

If the fields $\nf$ and $\nfup$ are global fields, such as number fields, then we call it a \textit{global embedding problem}.

Let $a$ be a non-negative integer. Let $\nfext$ be a cyclic extension of degree $2^a$. Denote $\galgp$ by $G$. Consider any central group extension of the form
$$
1 \to C_2 \to C_{2^{a+1}} \to \cgquo \to 1,
$$
and denote it by $\vareps_2$. Note that $\vareps_2$ is unique up to non-unique isomorphism.

\begin{ex}
For example, take $\nfext$ to be $\IQ(\sqrt{5})/\IQ$ and denote its Galois group by~$\grp$.
\begin{enumerate}
\item The embedding problem $\embprob[2]$ has $\IQ(\zeta_5)$ as a solution.
\item The embedding problem $\embprob$ in which $\vareps$ is of the form $$1~\to~C_2~\to~C_2~\times~C_2~\to~G~\to~1$$ has $\IQ(\sqrt{5}, i)$ as a solution. \qedhere
\end{enumerate}
\end{ex}

A global embedding problem $\embprob$ has one or more associated local embedding problems for each place of $L$ as follows.

\begin{defn}
Let $\embprob$ be a global embedding problem where $\vareps$ is an exact sequence $1 \to \cgnor \too{\iota} \cgext \too{\pi} \cgquo \to 1$.
Let $\plup$ be a place of $\nfup$ over a place $\pl$ of $\nf$
and denote by $\widetilde{\cgquo}$ the decomposition group $\decgp$,
which is the Galois group of $\lfext$.
Let $\widetilde{\cgext}$ be a subgroup of $\cgext$ such that
\begin{equation} \label{eq:relation}
\pi(\widetilde{\cgext}) = \widetilde{\cgquo}.
\end{equation}
Let $\widetilde{\cgnor} = \iota^{-1}({\widetilde\cgext})$ and denote by $\widetilde{\vareps}$ the following exact sequence
$$
1 \to \widetilde{\cgnor} \too{\iota} \widetilde{\cgext} \too{\pi} \widetilde{\cgquo} \to 1.
$$
Then $(\lfext, \widetilde\vareps)$ is the \textit{local embedding problem induced by the global embedding problem $\embprob$ with respect to the place $\plup$ and the subgroup $\widetilde{\cgext}$ of $\cgext$}.
\end{defn}

The following lemma gives a sufficient condition to conclude that a global embedding problem has no solution.

\begin{lem}[name={Richter, \cite[Satz 5]{richter1936}}]
If a global embedding problem $\embprob$ is solvable, then for each place $\plup$ of $\nfup$ there exists a subgroup $\widetilde\cgext$ of $\cgext$ such that the local embedding problem with respect to $\plup$ and $\widetilde\cgext$ is solvable.
\end{lem}

\begin{ex}
Let $\nfext$ be $\IQ(\sqrt{5})/\IQ$ with Galois group $\galgp$. Consider the embedding problem $\embprob[2]$. Here $\vareps_2$ is the exact sequence $1 \to C_2 \to C_4 \to C_2 \to 1$. Let $\plup$ a real place of $\IQ(\sqrt{5})$ over the unique (real) archimedean place $\pl$ of $\IQ$. Note that $\lfup = \IR$ and $\lf = \IR$ and the decomposition group $\decgp = \widetilde{\cgquo}$ is trivial. The subgroups of $C_4$ which satisfy \eqref{eq:relation} are exactly the trivial group and the unique subgroup of order $2$.
\begin{enumerate}
  \item The field $M = \IR$ is a solution to the local embedding problem induced by the global embedding problem $\embprob[2]$ with respect to the place $\plup$ and the trivial subgroup of $C_4$ since $\Gal(M/\IR) \cong 1$.
  \item The field $M = \IC$ is a solution to the local embedding problem induced by the global embedding problem $\embprob[2]$ with respect to the place $\plup$ and the unique subgroup $C_2$ of order $2$ of $\cgext$ since $\Gal(M/\IR) \cong C_2$. \qedhere
\end{enumerate}
\end{ex}

\begin{ex} \label{ex:locnonsolvable}
Let $\nfext$ be $\IQ(\sqrt{-5})/\IQ$. Let $\plup$ be the complex place of $\nfup$, which is above the unique (real) archimedean place $\pl$ of $\IQ$. The decomposition group $\decgp$ in this case is of order $2$. Consider the embedding problem $\embprob[2]$ where $\vareps_2$ is the exact sequence $1 \to C_2 \to C_4 \to C_2 \to 1$. Take $\widetilde\cgext = \cgext = C_4$, and note that this is the only subgroup of $\cgext$ which satisfies \eqref{eq:relation}. There does not exist a number field $M'$ such that $\Gal(M'/\IR) \cong \widetilde\cgext = C_4$. So, this induced local embedding problem is not solvable. Moreover, since $C_4$ is the only subgroup of $\cgext$ satisfying \eqref{eq:relation}, this is the only induced local embedding problem and hence all induced local problems are not solvable. As all valid candidates of $\widetilde{\cgext}$ result in a local problem which is not solvable, then there does not exist a cyclic field extension of $\IQ$ of degree $4$ which contains $\IQ(\sqrt{-5})$.
\end{ex}
\begin{ex}\label{ex:locinf}
If $\nf$ has no real embeddings, then all its archimedean places are complex and hence $\widetilde{G}$ is always trivial. In this case, taking the trivial group is the only valid choice for $\widetilde{\cgext}$. Hence, for each \textit{archimedean} place $\plup$ of $\nfup$, the global embedding problem $\embprob$ induces a local embedding problem with respect to $\plup$ which is solvable.
\end{ex}

We are mainly interested in the case where $\nfext$ is unramified. The following lemma shows that in this case, for each \textit{nonarchimedean} place $\plup$ of $\nfup$, the global embedding problem $\embprob$ induces a local embedding problem with respect to $\plup$ which is solvable.

\begin{lem}[label={lem:locfin},name={\cite[Satz 6]{richter1936}}]
Let $\ell, m, n, u$ be positive integers. Let $K$ be a nonarchimedean local field of characteristic~$0$ with unique prime ideal $\frakp$. Suppose that $K$ contains the $\ell^u$-th roots of unity, but not all $\ell^{u+1}$st roots of unity. Let $L$ be a cyclic extension of $K$ of degree $\ell^n$. Then, there exists a Galois extension $M$ of $K$ containing $L$ such that $\Gal(M/K) = C_{\ell^{m+n}}$ if and only if at least one of the following is true:
\begin{enumerate}
\item $\frakp$ is unramified in $L$.
\item $\frakp \nmid \ell$, and $u \geq m + s$, where $\ell^s$ is the ramification index of $\frakp$ in $L$
\item $\frakp \mid \ell$, and one of the following is true
\begin{itemize}
\item $u = 0$
\item $u \geq n + m$
\item $0 < u < n + m$ and $\zeta_{\ell^{\min(u, m)}} \in N_{L/K}(L)$.
\end{itemize}
\end{enumerate}
\end{lem}

Finally, we conclude by a lemma stating that a \textit{local-global principle} for our case.

\begin{lem}[label={lem:locglob}]
For any non-negative integer $a$ and any cyclic field extension $\nfext$ of degree $2^a$, the global embedding problem $\embprob$ is solvable if and only if for every place $\plup$ of $\nfup$, the unique induced local embedding problem is solvable.
\end{lem}

\begin{proof}
This is a special case of \cite[Satz 9]{richter1936} obtained by substituting $\ell, m$, and $n$ with $2, 1$, and $a$ respectively and noticing that the condition $B(2)$, defined in \cite[Definition 3]{richter1936}, is trivially satisfied.
\end{proof}

Finally, we end this subsection with a proof of \Cref{lem:globemb}.

\begin{proof}[label={prf:globemb},name={of \Cref{lem:globemb}}]
We are interested in the solvability of the embedding problem $\embprob$ where $\vareps$ is of the form $1 \to C_2 \to C_{2^{a+1}} \to G \to 1$, where $G = \galgp = C_{2^a}$. If we show that the global embedding problem is solvable, then we will have proven the lemma.
Since $\nf$ has no real embeddings, each archimedean place has a local embedding problem which is solvable thanks to \Cref{ex:locinf}. Now, since $\nfext$ is unramified, we may use \Cref{lem:locfin} to show that each nonarchimedean place has a local embedding problem which is solvable. Finally, using \Cref{lem:locglob}, we find that since each place of $\nf$ has an induced local embedding problem which is solvable, then the global embedding problem is solvable.
\end{proof}

  \subsection{Towards an explicit \texorpdfstring{$m$}{m}.}
  \label{subsec:embprobbounds}
\writefilename

\renewcommand{\nf}{K}%


The following result of Crespo is one of the key ingredients in the proof of our main result.

\begin{thm}[label={thm:crespo6},name={\cite[Theorem 6]{crespo1989}}]
Let $\nfext$ be a Galois extension of a number field $\nf$, unramified outside a finite set $\primeset$ of prime ideals of the ring of integers $\Onf$ of $\nf$. Let $n$ be a positive integer, $G = \galgp$, and $A$ an abelian group of exponent $n$. Assume $\primeset$ contains the prime ideals dividing $n$. For each prime number $p$ dividing $n$, we denote by $a_p$ the $p$-rank of $A$, by $r_p$ the $p$-rank of $\Hom(G, A)$ and let $\delta_p = 0$ if $K$ contains a primitive $p^{v_p(n)}$-th root of unity and $\delta_p = 1$ if it does not. Suppose that
\begin{enumerate}
  \item the order $h_\primeset$ of the $S$-class group is coprime to $n$, and
  \item for every prime number $p \mid n$, we have $r_p + a_p + \delta_p < \#\primeset$.
\end{enumerate}
Then every solvable embedding problem $\embprob$, where $\vareps$ is a central group extension of $G$ by $A$, has a solution $\nfupup$ such that $\nfextbig$ is unramified outside $\primeset$.
\end{thm}


Given a finite abelian extension $\nfext$,
we denote its conductor, as defined in \cite[Chapter 2]{cohen2000}, by $\cond{\nfext}$.
One key property of the conductor that we use is that it is the minimal modulus $\frakm$ such that $\Hnf[\frakm] \supseteq \nfup$.

\begin{lem}\label{lem:globembcond}
Let $a$ be a non-negative integer. Let $\nf$ be a number field with no real embeddings.
Let $\nfext$ be an unramified cyclic Galois extension of degree $2^a$. 
Let $\primeset$ be a finite set of prime ideals of $\nf$ such that
\begin{itemize}
\item $| \Cl_\nf(1) / \langle \primeset \rangle |$ is odd,
\item $\primeset$ contains all prime ideals above $2$,
\item $\primeset$ contains at least $3$ elements.
\end{itemize}
Then there exists a cyclic Galois extension $\nfextbig$ of degree $2^{a+1}$, unramified outside $\primeset$, containing $\nfup$.
\end{lem}

\begin{proof}
\Cref{lem:globemb} shows that the embedding problem $\embprob[2]$ is solvable. Keeping the notation of \Cref{thm:crespo6}, the $2$-rank $a_2$ of $\cgnor = C_2$ for this embedding problem is $1$. Moreover, $\Hom(\galgp, C_2) \cong C_2$ and hence $r_2 = 1$. Finally $\delta_2 = 0$ since $\nf$ contains the second roots of unity. Using \Cref{thm:crespo6}, we prove the lemma.
\end{proof}


Denote by $\reld{\nfext}$ the relative discriminant ideal of a field extension $\nfext$, as defined in \cite[Chapter 2, Section 2.4]{cohen2000} and in \cite[Section III.2.8]{neukirch1999}.

We now state the following lemma.

\begin{lem}[label={lem:cohen3321},name={Cohen, \cite[Proposition 3.3.21]{cohen2000}}]
Let $\nfext$ be an abelian extension of degree $n$ such that $\nfup \subseteq \Hnf[\frakm]$ for some modulus $\frakm$. Let $\frakp$ be a prime ideal of $\calO_K$ such that $\cond{\nfext}(\frakp) \neq 0$. Finally, let $\ell$ be the prime number below $\frakp$.
\begin{enumerate}
  \item If $\ell \nmid n$, then $\cond{\nfext}(\frakp) = 1$.
  \item If $\text{gcd}(n, N_\nfext(\frakp)-1) = 1$ and $n$ is a power of $\ell$, then $\cond{\nfext}(\frakp) \geq 2$.
\end{enumerate}
\end{lem}


From \Cref{lem:cohen3321}, we conclude that since $2$ is the only prime divisor of $[\nfupup : \nfup]$, with $\nfupup, \nfup$ as in \Cref{lem:globembcond}, then for a prime ideal $\frakP$ of $\nfup$ not above $2$, we have $v_\frakP(\cond{\nfextup}) \leq 1$.

Using \cite[Corollary 10.1.24]{cohen2000} gives us the bound $\cond{\nfextup}(\frakP_2)~\leq~2e(\frakP_2/2)~+~1$ for any $\frakP_2$ above $2$, where $e(\frakP_2/2)$ is the ramification index of $\frakP_2$ over $2$.
To summarize, for any prime ideal $\frakP$ of $\Onfup$, we have:
\begin{equation} \label{eq:bdoncondf}
\cond{\nfextup}(\frakP) \leq \begin{cases}
2e(\frakP/2) + 1 & \frakP \mid 2 \\
1 & \frakP \nmid 2. \\
\end{cases}
\end{equation}
\Cref{prop:valbound}, below, enables us to bound the valuation of $\cond{\nfextbig}$ at the primes $\frakp$ of $\Onf$ using the bounds on the valuations of $\cond{\nfextup}$ at the primes $\frakP$ of $\Onfup$.
\begin{prop} \label{prop:valbound}
Let $K$ be a number field and let $L$ be an unramified extension of $K$ of degree $2^a$. Let $M$ be a cyclic extension of $K$ of degree $2^{a+1}$ which contains $L$. Let $\frakp$ be an ideal of $\Onf$. Let $c$ be an integer and suppose $\cond{\nfextup}(\frakP) \leq c$ for every prime ideal $\frakP$ of $\Onfup$ above $\frakp$. Then
\[
\cond{\nfextbig}(\frakp) \leq c.
\]
\end{prop}
\begin{proof}
Note that
$$
\ord_\frakp(N_{\nfext}(\cond{\nfextup}))
=
\sum_{\frakP|\frakp} \ord_\frakp(N_{\nfext}(\frakP)) \cdot \ord_{\frakP}(\cond{\nfextup}),
$$
where $\sum_{\frakP | \frakp}$ denotes a sum that runs through all primes $\frakP$ over $\frakp$.
Using the assumption that $\cond{\nfextup}(\frakP) \leq c$ for every prime ideal $\frakP$ of $\Onfup$ above $\frakp$, we get
$$
\ord_\frakp(N_{\nfext}(\cond{\nfextup}))
\leq
c \cdot \sum_{\frakP|\frakp} \ord_\frakp(N_{\nfext}(\frakP)).
$$
Let $g$ be the number of prime ideals $\frakP$ of $\Onfup$ above $\frakp$. For each of these $g$ prime ideals, the norm $N_{\nfext}(\frakP)$ is given by the residue class degree $f = [\Onfup / \frakP : \Onf / \frakp ]$. Hence, we have $\sum_{\frakP|\frakp} \ord_\frakp(N_{\nfext}(\frakP))
=
fg.$
Now, since $\nfext$ is unramified, we have $2^a = [\nfup : \nf] = fg$. Corollary III.2.10 of \cite{neukirch1999} states that for a tower of fields $\nf \subseteq \nfup \subseteq \nfupup$ one has
\begin{equation} \label{eq:part1}
\reld{\nfextbig}
=
\reld{\nfext}^{[M:L]}
N_\nfext(\reld{\nfext}).
\end{equation}
The conductor-discriminant formula \cite[Section VII.11.9]{neukirch1999} gives us
\begin{equation} \label{eq:part2}
\reld{\nfextup} = \cond{\nfextup},
\hspace{3em}
\reld{\nfextbig} = \cond{\nfextbig}^{2^a}
\end{equation}
Combining \eqref{eq:part1}, \eqref{eq:part2} and the fact that $\reld{\nfext} = 1$ since $\nfext$ is unramified,
we obtain
$$
\cond{\nfextbig}^{2^a} = N_{\nfext}\left(\cond{\nfextup}\right).
$$
And thus
\[
2^a \cdot \ord_\frakp(\cond{\nfextbig})
=
\ord_\frakp(N_{\nfext}(\cond{\nfextup}))
\leq
2^a \cdot c
\]
and so
$
\ord_\frakp(\cond{\nfextbig}) \leq c.
$
\end{proof}

\begin{tmpvar} \newcommand{\tmpnf}{K}%
For each number field $\tmpnf$ and for each integer $\frakm$, let
$E_\tmpnf(\frakm)$ be the smallest subfield of $H_\tmpnf(\frakm)$ containing $\tmpnf$
such that $\Gal(H_\tmpnf(\frakm) / E_\tmpnf(\frakm))$ is of exponent at most $2$.
\end{tmpvar}

\begin{thm} \label{thm:main1withideals}
Let $K$ be a number field without real embeddings. Let $\primeset$ be a finite set of prime ideals of $\Onf$ such that
\begin{itemize}
\item $| \Cl_K(1) / \langle \primeset \rangle |$ is odd,
\item $\primeset$ contains all prime ideals above $2$,
\item $\primeset$ contains at least $3$ elements.
\end{itemize}
Let $\frakmS = 4 \cdot \prod_{\frakp \in \primeset} \frakp$. Then $H_K(1) \subseteq E_K(\frakmS)$.
\end{thm}

\begin{proof}
Suppose $\Gal(H_K(1)/K)$ is
$$
\Gal(H_K(1)/K)
=
G_0
\times
G_1
\times
\cdots
\times
G_t
$$
where $G_0$ is the largest subgroup of
$\Gal(H_K(1)/K)$ of odd order,
and $G_i$ is a cyclic group of order $2^{a_i}$ generated by $\sigma_i$
for $i \in \{1, \ldots, t\}$.
For each $j \in \{0, 1, \ldots, t\}$, let $L_j$ be the fixed field of $$G_0 \times \cdots \times G_{i-1} \times \langle 1 \rangle \times G_{i+1} \times \cdots \times G_t$$
by Galois theory.
Fix an $i \in \{1, \ldots, t\}$.
Since $L_i/K$ is an unramified cyclic number field extension of degree $2^{a_i}$,
\Cref{lem:globembcond} gives us the existence of a field extension $M_i$ of $K$ containing $L_i$ which is unramified outside $S$.
Let $\frakP_i$ be a prime ideal of $\calO_{L_i}$ above
a prime ideal $\frakp$ of $\Onf$,
and a rational prime $\ell$.
Since $L_i/K$ is unramified,
we find that
$
e(\frakP / \ell)
= 
e(\frakP / \frakp) e(\frakp / \ell)
= 
e(\frakp / \ell).
$
\Cref{eq:bdoncondf} and \Cref{prop:valbound} then tell us that
\[
\ord_{\frakp}(\frakf_{M_i/K}) \leq \begin{cases}
2e(\frakp/\ell) + 1 & \ell = 2 \\
1 & \ell \neq 2. \\
\end{cases}
\]
Since the conductor of a compositum of fields
divides the least common multiple of
the conductors of the fields being composed, the field $L_0M_1 \cdots M_t$ has a conductor
which divides
\[
\frakm
=
\prod_{\frakp \mid 2} \frakp^{2e(\frakp/2)}
\prod_{\frakp \in S} \frakp
=
4 \prod_{\frakp \in S} \frakp.
\]
{\newtmpvar{\tmpintfld}{K'}%
Denote by $\galgpHnf{\tmpintfld}$ the Galois group $\Gal(\Hnf[\frakm] / \tmpintfld)$
where $\tmpintfld$ is an abelian extension of $\nf$ contained in $\Hnf[\frakm]$.
}%
We want to show that
$H_K(1) \subseteq E_K(\frakm)$.
To do this, we show the equivalent condition
$$
\galgpHnf{E_K(\frakm)} \subseteq \galgpHnf{H_K(1)}.
$$
Let $\sigma \in \galgpHnf{E_K}(\frakm)$ and note that $\sigma^2 = 1$.
Note that $[\galgpHnf{K} : \galgpHnf{L_0}]$ is an odd integer and hence $\sigma \in \galgpHnf{L_0}$.
On the other hand,
for each $i \in \{1, \ldots, t\}$,
since $\sigma^2 = \id \in \galgpHnf{M_i}$
then by
definition of $M_i$,
$\sigma \in \galgpHnf{L_i}$.
Hence $\sigma$ fixes $L_0L_1 \cdots L_t = H_K(1)$.
And therefore $\sigma \in \galgpHnf{H_K(1)}$.
\end{proof}

The smallest positive integer $m$ contained in $\frakm$ is given by $m = 4P$ where $P$ is the product of all primes $p$ such that $p$ is below some $\frakp \in S$. With this observation and the fact that $E_\cmf(\frakm) \subseteq E_\cmf(\frakn)$ when $\frakm \mid \frakn$, \Cref{thm:main1} becomes a direct consequence of \Cref{thm:main1withideals}.

\section{Given an integer \texorpdfstring{$m$}{m}, does \texorpdfstring{$\ref{eq:star}$}{(*m)} hold?}
\label{sec:decisionprob}
\writefilename

Let $\cmfcmt$ be a primitive CM pair
and let $\cmfrcmtr$ be its reflex pair.

The goal of this section is to
describe an algorithm that,
given a positive integer $m$,
outputs whether or not \eqref{eq:star} holds.

From this section onwards,
denote by $\genord{g}{e}$
the cyclic group generated by an element $g$ of order $e$.
Note that $e$ may be $\infty$.

As the fields $\Hcmfr[1]$,
$\cmfr \Hcmfor[m]$,
$\CMcmfrcmtr[m]$, and $\Hcmfr[m]$
are all
abelian extensions of $\cmfr$
contained in $\Hcmfr[m]$,
we may use Galois theory
to rewrite \eqref{eq:star} 
in terms of subgroups
of the finite abelian group
$\Gal(\Hcmfr[m]/\cmfr)$
as%
\begin{equation} \label{eq:galstar}\tag{\ensuremath{\star\star_m}}
\galgpHcmfr{\Hcmfr[1]}
\,\supseteq\,
\galgpHcmfr{\cmfr \Hcmfor[m]}
\,\cap\,
\galgpHcmfr{\CMcmfrcmtr[m]},%
\end{equation} \renewcommand*{\theHequation}{notag.\theequation}%
\begin{tmpvar}\newcommand{\tmp}{K'}%
where $\galgpHcmfr{\tmp}$ is the subgroup of $\Gal(\Hcmfr[m]/\cmfr)$ fixing $\tmp$.%
\end{tmpvar}

As a subfield of $\Hcmfr[m]$,
the field $\Hcmfr[1]$ corresponds to the congruence subgroup
$$
\{
\fraka \in \Idcmfr[m]
:
\fraka = a\Ocmfr \text{ for some } a \in \cmfr
\}
$$
of $\Idcmfr[m]$.
The Galois group $\galgpHcmfr{\Hcmfr[1]}$ is
the kernel of the natural surjective map 
\begin{equation} \label{eq:f0}
\effzero: \Clcmfr[m] \to \Clcmfr[1].
\end{equation}
In the same vein,
the field $\cmfr \Hcmfor[m]$ corresponds to the congruence subgroup
$$
\{
\fraka \in \Idcmfr[m]
:
\fraka\bar\fraka = (a) \text{ for some } a \in \cmfor, \, a \equiv 1 \mod^\ast m
\}.
$$
Hence,
the Galois group $\galgpHcmfr{\cmfr\,\Hcmfor[m]}$ is also isomorphic to a kernel,
the kernel of the relative norm map
\begin{equation} \label{eq:f1}
\effone = \relnm: \Clcmfr[m] \to \Clcmfor[m].
\end{equation}

We can also compute the Galois group $\galgpHcmfr{\CMcmfrcmtr[m]}$ as a kernel of a map $\efftwo$ which we define in \Cref{subsec:shray}. The codomain of $\efftwo[1]$ is the Shimura class group studied in \cite[Section 3.1]{broker2011}. We generalize this Shimura class group by defining, in the same section, the Shimura \textit{ray} class group of $\cmf$ for a modulus $m$, which we denote by $\srcgcmf[\mdlsZ]$.

Let $\mdlsZ$ be a positive integer.
\Cref{subsec:shraycompute}
provides our algorithm to compute
the generators of
$\srcgcmf[\mdlsZ]$,
their respective orders as group elements,
and a discrete logarithm algorithm for
$\srcgcmf[\mdlsZ]$.
The end of the section details how
we can extract and use information about the group $\srcgcmf[\mdlsZ]$ to determine
whether or not \eqref{eq:galstar} holds
for the given integer $m$.

  \subsection{The Shimura ray class group}
  \label{subsec:shray}
\writefilename

Let $\cmfcmt$ be a CM pair and $\cmfrcmtr$ its reflex pair. The Shimura class group of $\cmf$, in conjunction with a group morphism involving the type norm map, was used in \cite[Section 2.2]{engethome2013} to compute the Galois group $\galgpHcmfr{\CMcmfrcmtr[1]}$. This section generalizes the Shimura class group and introduces the concept of a \textit{Shimura ray class group} for each modulus $\mdls$ of $\cmf$. We define it as follows.

\begin{defn}[Shimura ray class group for the modulus $\mdls$]
	Let $\mdls$ be a modulus of a CM field $\cmf$.
	The Shimura ray class group $\srcgcmf[\mdls]$ is the group given by
	\[
	\srcgcmf[\mdls]
	=
	\frac{\{ (\fraka, a) \in \Idcmf[\mdls] \times \cmfox : \ibari{\fraka} = a\Ocmf, a \gg 0 \}
	}{\{(x\Ocmf, \ibari{x}) \in \Idcmf[\mdls] \times \cmfox : x \in K^\times, x \equiv \onemodstar \mdls\}}.
	\]
Multiplication of elements in $\srcgcmf[\mdls]$ is done by component-wise multiplication.
\end{defn}

Notice that the definition of the Shimura ray class group of a CM field $\cmf$ is independent of CM types and also does not concern reflex fields. 

For any positive integer $m$ and any CM pair $\cmfcmt$,
the type norm map induces a map from the ray class group of $\cmfr$
to the Shimura ray class group of $\cmf$ as follows
\begin{equation} \label{eq:calNm}
\begin{aligned}
\efftwo: \Clcmfr[m] &\to \srcgcmf[m] \\
[\frakb] &\mapsto [\left(\typnmr(\frakb), \absnm{\Kr}[\frakb]\right)].
\end{aligned}
\end{equation}
The kernel of the map $\efftwo$,
by definition of $\Idcmfrcmtr[m]$,
is exactly  $\Idcmfrcmtr[m] / \Prcmfr[m]$
and is thus isomorphic to
$\galgpHcmfr{\CMcmfrcmtr[m]}$.

Just like the case when $\mdls = 1$, the Shimura \textit{ray} class group  for a modulus $\mdls$ fits as the `$B$'-term term of a short exact sequence $1 \to A \to B \to C \to 1$ for some computable $A$ and $C$. We start by introducing and computing the ingredients of $A$ and $C$.

Let $\cmfo$ be the real subfield of a CM field $\cmf$. Write $\Ocmfox$ for the group of units of the ring of integers of $\cmfo$ and write $\Ocmfoxp$ for the subgroup of $\Ocmfox$ consisting of only the totally positive units of $\cmfo$.

\begin{ex} \label{ex:Ocmfox}
Let $\cmf$ be a quartic CM field different from $\IQ(\zeta_5)$.
Dirichlet's unit theorem gives us
$$
\Ocmfox = \genord{-1}{2} \times \genord{\fuo}{\infty}
$$
for some fundamental unit $\fuo$.
Furthermore, we find that $\Ocmfoxp = \langle \fuo^+ \rangle$ where
$$
\fuo^+ = \begin{cases}
\fuo & \fuo \gg 0 \\
-\fuo & \fuo \ll 0 \\
\fuo^2 & \text{otherwise}. \\
\end{cases}
$$
\end{ex}

Denote by $\Ocmfonex$ the kernel of the natural map
\begin{equation} \label{eq:tildebullet}
s: \Ocmfx \to \left( \Ocmf / \mdls \right)^\times
\end{equation} 
The image $\relnm\left(\Ocmfonex\right)$ is easily observed to be contained in $\Ocmfoxp$.
Indeed, we have $\cmf = \cmfo(\sqrt{-z})$ for some totally positive element $z \in \cmfo$
and the relative norm of a nonzero element $x = a + b\sqrt{-z} \in \cmf$ is $a^2 + b^2z$,
which is a totally positive element.
Thus, the norm $\relnm: \cmf \to \cmfo$
induces maps
\begin{equation} \label{eq:N1}
\None := N_{K/K_0}: \Ocmfonex \to \Ocmfoxp
\qquad : \qquad
x \mapsto \ibari{x}.
\end{equation}
and
\begin{equation} \label{eq:N2}
\Ntwo := \relnm: \Clcmf[\mdls] \to \Clcmfo[1][+]
\qquad : \qquad
[\fraka] \mapsto [\ibari{\fraka}].
\end{equation}

We define the maps $f$ and $g$ as follows
\begin{aligncol*}{3}
f: \Ocmfoxp &\to \srcgcmf[m] \\
u &\mapsto [(\Ocmf, u)] \\
&\text{and}\\
&\\
g: \srcgcmf[m] &\to \Clcmf[\mdls] \\
[(\fraka, a)] &\mapsto [\fraka].
\end{aligncol*}

We are now ready to state and prove the following lemma.
\begin{lem} \label{lem:shortexactlemma}
	The sequence
	\begin{equation*}
\Ocmfonex
\too{\None}
\Ocmfoxp
\too{f}
\srcgcmf[\mdls]
\too{g}
\Clcmf[\mdls]
\too{\Ntwo}
\Clcmfo[1][+]
	\end{equation*}
	is exact. Consequently, the sequence
\begin{equation*}
1
\to
\coker \None
\too{f}
\srcgcmf[\mdls]
\too{g}
\ker Ntwo
\to
1
\end{equation*}
is exact.
\end{lem}

\begin{proof}
	We first prove exactness at $\Ocmfoxp$.
	Let $u \in \Ocmfoxp$ such that $[(\Ocmf, u)]$ is trivial in $\srcgcmf[m]$.
	Since $[(\Ocmf, u)]$ is the trivial class, the unit $u$ is of the form $\ibari{x}$ where $x \equiv \onemodstar \mdls$.
	Hence $u = \ibari{x}$ for some $x \in \Ocmfonex$.
	Thus, $\ker f \subseteq \im \None$.
	Moreover, for $x \in \Ocmfonex$, we have $f(\None(x)) = [(\Ocmf, \ibari{x})]$.
	This element is trivial in $\srcgcmf[m]$.
	Hence $\im \None \subseteq \ker f$.
	
	We prove exactness at $\srcgcmf[\mdls]$.
	Given $[(\fraka, a)] \in \ker g$,
	we have $\fraka = \alpha\Ocmf$ for some $\alpha \in K$
	with $\alpha \equiv 1 \mod^\star \mdls$.
	And so $\ibari{\alpha}\Ocmf = \ibari{\fraka} = a\Ocmf$.
	Hence $\ibari{\alpha}u = a$ for some unit $u \in \Ocmf$.
	Since $\ibari{\alpha}$ is a relative norm
	for the extension $K/\cmfo$, it is in $\cmfo$
	and totally positive.
	Moreover, tle element $a$ is also in $\cmfo$ and totally positive by definition of $\srcgcmf[m]$.
	Thus $u$ must also be in $\cmfo$ and totally positive and hence $u \in \Ocmfoxp$.
	Since
	\[
	[(\fraka, a)]
	=
	[(\alpha\Ocmf, \ibari{\alpha} u)]
	=
	[(\Ocmf, u)],
	\]
	the class $[(\fraka, a)]$ is evidently in the image of the $f$.
	And so $\ker g \subseteq \im f$.
	Now, for any $u \in \Ocmfoxp$,
	we have
	$g(f(u)) = [\Ocmf]$.
	Hence, $\im f \subseteq \ker g$.
	
	We prove exactness at $\Clcmf[\mdls]$.
	Suppose $[\fraka] \in \Clcmf[\mdls]$ is such that $[\ibari{\fraka}] = a\Ocmfo$ for some $a \in \cmfo$ with $a$ totally positive.
	And so $[\fraka, a] \in \srcgcmf[m]$ and $g([\fraka, a]) = [\fraka]$.
	Hence, $\ker \Ntwo \subseteq \im g$.
	Suppose $[(\fraka, a)] \in \srcgcmf[m]$.
	First $g([(\fraka, a)]) = [\fraka]$.
	By definition of $\srcgcmf[m]$, we have $\Ntwo([\fraka]) = a\Ocmfo$ for some $a \in \cmfo$ with $a$ totally positive. Thus $\Ntwo(g([(\fraka, a)]))$ is trivial. Thus $\im g \subseteq \ker \Ntwo$.
\end{proof}

  \subsection{Computing the Shimura Ray Class Group}
  \label{subsec:shraycompute}
\writefilename

\Cref{lem:shortexactlemma} states that
$\srcgcmf[\mdls]$ fits as the `$B$-term' in a short exact sequence of the form $1 \to A \to B \to C \to 1$.
In this section, we compute each term in that short exact sequence using algorithms on finitely generated abelian groups found in \cite{cohen2001} and \cite[Chapter 4]{cohen2000}. 
All algorithms discussed in this section are efficient and practical in the sense that
\begin{itemize}
\item \label{enum:ALG1} fast implementations are available\footnote{Available either as one of the built-in functions, or implemented by the author.} in PARI/GP \cite{pari211} or Magma \cite{magma}, and
\item \label{enum:ALG2} when these implementations are used for our examples in \Cref{subsec:algoimpl}, they are in practice not the dominant step of the computation.
\end{itemize}

Let $G$ be a finitely generated abelian group.
By the fundamental theorem of finitely generated abelian groups,
there exist unique non-negative integers $1 \neq \ordvar_1, \ldots, \ordvar_\gencnt$ such that $\ordvar_\gencnt \mid \ordvar_{\gencnt-1} \mid \cdots \mid \ordvar_1$
and there exists \textit{a discrete logarithm isomorphism}
$$
\dlogiso{\grp} : \grp \to \grpdlg :=
\genordadd{}{\ordvar_1}
\times
\cdots
\times
\genordadd{}{\ordvar_\gencnt}
$$
Given such an isomorphism $\dlogiso{\grp}$, we say that the \textit{standard set of generators} with respect to $\dlogiso{\grp}$ is the set of inverse images of all $\genvarstd_i \in \grpdlg$, where $\genvarstd_i$ is the vector in $\grpdlg$ whose entries are all $0$, save for the $i$th entry, which is $1$.

We say that a finitely generated abelian group $\grp$ is \textit{computed} if the integers $d_1, \ldots, d_n$ of $V_G$ are known and there exists a discrete logarithm isomorphism $\dlogiso{\grp}$ from $\grp$ to $\grpdlg$
such that
\begin{enumerate}[leftmargin={3em},leftmargin={5em},label=\textbf{CG\arabic*}]
\item \label{enum:CG1}%
the inverse images under $\dlogiso{\grp}$ of the generators $\genvarstd_1, \ldots, \genvarstd_\gencnt$ are known, and
\item \label{enum:CG2}%
an algorithm that outputs $\dlogiso{\grp}(g)$ given $g \in \grp$ is known.
\end{enumerate}

A subgroup $\subgrp$ of a computed group $\grp$
is said to be \textit{computed with respect to $\dlogiso{\grp}$}
if we know an $n \times n$ matrix
$\hat{\subgrp} = (h_{i,j})$
in column Hermite normal form (HNF)
such that
$$
\left\{
\dlogiso{\grp}^{-1}
\left(\trpos*{\begin{bmatrix}
h_{1,j} & \cdots &h_{n,j}
\end{bmatrix}}\right)
:
j = 1, \ldots, \gencnt
\right\}
$$
generates $\subgrp$.
With such a Hermite normal form matrix, one can use a Smith normal form (SNF) algorithm \cite[Algorithm 4.1.3]{cohen2000} to compute $H$ as in \ref{enum:CG1} and \ref{enum:CG2}.

Let $\cmf$ be a quartic CM field
and $\mdls$ be a modulus of $\cmf$.
From \Cref{lem:shortexactlemma},
the group
$\srcgcmf[\mdls]$
fits as the middle term of the exact sequence
\begin{equation} \label{lem:shortexactCKm}
1
\to
\coker \None
\too{f}
\srcgcmf[\mdls]
\too{g}
\ker \Ntwo
\to
1.
\end{equation}

We would like to compute this middle term. We use the group extension algorithm \cite[Algorithm 4.1.8]{cohen2000} to compute $\srcgcmf[\mdls]$. To compute the middle term of \eqref{lem:shortexactCKm} using this algorithm, we need:
\begin{enumerate}[leftmargin={3em},leftmargin={5em},label=\textbf{GEXT\arabic*}]
\item \label{enum:GEXT1} to compute the group $\ker \Ntwo$,
\item \label{enum:GEXT2} to compute the group $\coker \None$,
\item \label{enum:GEXT3} an algorithm which, given $[\frakb, \beta] \in \im f$, outputs $a \in \coker N_1$ such that $f(a) = [\frakb, \beta] \in \im f$.
\end{enumerate}

We first compute the kernel $\ker \Ntwo: \Clcmf[\mdls] \to \Clcmfo[1][+]$ using the inverse image algorithm found in \cite[Algorithm 4.1.11]{cohen2000}. In order to use this algorithm, we need:
\begin{enumerate}[leftmargin={3em},leftmargin={5em},label=\textbf{KER\arabic*}]
\item \label{enum:KER1} to compute the groups $\Clcmf[\mdls]$ and $\Clcmfo[1][+]$, and
\item \label{enum:KER2} an algorithm which gives $\dlogiso{\Clcmfo[1][+]}(\Ntwo([\fraka]))$ given $[\fraka] \in \Clcmf[\mdls]$.
\end{enumerate}

The groups mentioned in \ref{enum:KER1} can be computed using \cite[Algorithm 4.3.1]{cohen2000}. For \ref{enum:KER2}, if we choose an ideal class representative $\fraka$ in $\Clcmf[\mdls]$ and apply the relative norm map $\relnm: \Idcmf[\mdls] \to \Idcmfo[1][+]$ on ideals using \cite[Algorithm 2.5.2]{cohen2000}, then $\relnm[\fraka]$ is a representative of the image of the ideal class $[\fraka]$ under $\Ntwo$. Hence, we can compute~$\dlogiso{\Clcmfo[1][+]}(\Ntwo([\fraka]))$ by using \cite[Algorithm 4.3.2]{cohen2000}.

We compute the quotient group
$$
\coker \None = \Ocmfoxp / \relnm(\Ocmfonex)
$$
using \cite[Algorithm 4.1.7]{cohen2000}. In order to use this algorithm, we need:
\begin{enumerate}[leftmargin={3em},leftmargin={5em},label=\textbf{QUO\arabic*}]
\item \label{enum:QUO1} to compute the group $\Ocmfoxp$, and
\item \label{enum:QUO2} to compute $\relnm(\Ocmfonex)$ with respect to the discrete logarithm isomorphism $\dlogiso{\Ocmfoxp}$ obtained from doing \cref{enum:QUO1}
\end{enumerate}

For \ref{enum:QUO1}, we simply compute $\Ocmfox$ using \cite[Algorithm 6.5.7]{cohen1993} and then use \Cref{ex:Ocmfox} to obtain generators and a discrete logarithm isomorphism $\dlogiso{\Ocmfoxp}$ for $\Ocmfoxp$. With this, for any $x \in \Ocmfox$, we can find $\dlogiso{\Ocmfoxp}(s)$ if we know $\dlogiso{\Ocmfox}(s)$.

For \ref{enum:QUO2}, we first compute the groups $\Ocmfx$ and $\idealstar[\mdls]$ using \cite[Algorithm 4.2.21]{cohen2000} and \cite[Algorithm 4.1.11]{cohen2000}, respectively. Let 
$$
s: \Ocmfx \to \idealstar[\mdls]
$$
be the map induced by the natural map $\Ocmf \to \Ocmf/\mdls$. We are interested in $\Ocmfonex$, which is $\ker s$. An algorithm to compute $\dlogiso{\idealstar[\mdls]}(s(x))$ which takes as input $x \in \Ocmfx$ is given by \cite[Algorithm 6.5.7]{cohen1993}. With such an algorithm and the fact that we have computed both $\Ocmfx$ and $\idealstar[\mdls]$, we can use \cite[Algorithm 4.1.11]{cohen2000} to compute $\Ocmfonex$ with respect to $\dlogiso{\Ocmfx}$. From there, we can use \cite[Algorithm 4.1.11]{cohen2000} to compute the image $\relnm(\Ocmfonex)$ with respect to $\dlogiso{\Ocmfoxp}$.

\begin{ex}[label={ex:runningex}]
Let $\cmfo$ be the real quadratic field $\cmfodab{\alpha_0}[809]{53}{500}$. Consider the quartic CM field $\cmf = \cmfo(\alpha)$ where $\alpha$ is a root of $X^2 - \alpha_0$.
Solving for fundamental units of $\Ocmfx$ and $\Ocmfox$, we get
$$
\vareps = 30506849866\alpha^2 + 374579495409$$
and
$$\vareps_0 = 30506849866\alpha_0 + 374579495409,$$
respectively.
Having $\vareps = \vareps_0$ corresponds to one of the three possible cases as discussed in \Cref{ex:Ocmfox}. Using the notation we established in \Cref{sec:decisionprob}, we have
$$
\Ocmfx = 
\genord{-1}{2}
\times
\genord{\vareps}{0}
\quad\quad
\text{and}
\quad\quad
\Ocmfox = 
\genord{-1}{2}
\times
\genord{\vareps_0}{0}.
$$
Meanwhile, the group $\idealstar[2\Ocmf]$ is given by
$$
\idealstar[2\Ocmf]
=
\genord{-1/10\alpha^3 - \alpha^2 - 33/10\alpha - 26}{2}
\times
\genord{-\alpha^2 -\alpha - 27}{2}.
$$
Since
$$
-1 \equiv 1 \mod 2
\quad\quad
\text{and}
\quad\quad
\vareps = 30506849866\alpha^2 + 374579495409\equiv 1 \mod 2,
$$
we find that the map $s: \Ocmfx \to \idealstar[2\Ocmf]$ sends all elements of its domain to $1$. Hence, $\Ocmfonex[2] = \Ocmfx$. The image of $\Ocmfonex[2]$ under the relative norm map $\relnm$ is the index $4$ subgroup
\[
\relnm\left( \Ocmfonex[2] \right) = \langle \vareps_0^2 \rangle \subseteq \Ocmfox.
\]
Note that $\vareps_0$ is not totally positive since the norm $\absnm{\cmf_0}(\vareps_0)$ of $\vareps_0$ is $-1$.
Thus, we find that
\[
\Ocmfoxp = \langle \vareps_0^2 \rangle
\]
And so $\coker \None$ is the trivial group.

We now compute $\ker \Ntwo$. The ray class groups $\Clcmf[2]$ and $\Clcmfo[1][+]$ are given by
$$
\Clcmf[2] = 
\genord{[\fraka_1]}{8}
\times
\genord{[\fraka_2]}{4}
\hspace{2em}
\text{where}
\hspace{2em}
\fraka_1 = (7, \alpha - 2)
\hspace{1em}
\text{and}
\hspace{1em}
\fraka_2 = (7159, \alpha - 2627),
$$
and
$$
\Clcmfo[1][+] = 1. 
$$
The codomain is trivial and so $\ker \Ntwo = \Clcmf[2]$.
\end{ex}

We now give an algorithm that \ref{enum:GEXT3} asks for. This algorithm is based on the remark below \cite[Algorithm 4.1.8]{cohen2000}.

\begin{algo} \label{algo:preimf} Finding inverse images of $f: \coker \None \to \srcgcmf[\mdls]$.
\\\textsc{Input}. $[\frakb, \beta] \in \im f$
\\\textsc{Output}. $a \in \coker \None$ such that $f(a) = [\frakb, \beta] \in \im f$
\begin{enumerate}
	\item Compute $\Clcmf[m]$ using \cite[Algorithm 4.3.1]{cohen2000}.
	\item Find $x$ such that $\frakb = x\Ocmf$ using \cite[Algorithm 4.3.2]{cohen2000}. This algorithm uses the computation from the previous step.
	\item Let $\alpha = \beta/(\ibari{x})$.
	\item Return $[\alpha]$.
\end{enumerate}
\end{algo}

If one wishes to take inverse images of multiple elements in $\im f$, one can compute $\Clcmf[m]$ once and for all as it does not depend on the input $[\frakb, \beta]$ and proceed with the second step.

Now that we have done \ref{enum:GEXT1}, \ref{enum:GEXT2}, and \ref{enum:GEXT3}, we use \cite[Algorithm 4.1.8]{cohen2000} and the paragraph below it to compute $\srcgcmf[\mdls]$ in the sense of \ref{enum:CG1} and \ref{enum:CG2}.

  \subsection{Using the Shimura ray class group to determine whether \texorpdfstring{$m$}{m}, does \texorpdfstring{$\ref{eq:star}$}{(*m)} holds.}
  \label{subsec:containment}
\writefilename

Let $(K, \Phi)$ be a primitive quartic CM pair and let $\cmfrcmtr$ be its reflex.
Let $m$ be a positive integer.

We have established in \Cref{subsec:shray} that $\ker \efftwo[m]$
is exactly $\Id{\Kr, \cmtr}[m] / \Pr{\Kr}[m]$,
which is isomorphic to $\Gal(\hcf{\Kr}[m] / \CMcmfrcmtr[m])$ via the Artin map.
Recall the functions $\effzero$ and $\effone$ defined in \eqref{eq:f0} and \eqref{eq:f1}, respectively. We have the following isomorphisms via the Artin map:
\begin{align} \label{eq:kereffs}
\begin{split}
\Gal(\hcf{\Kr}[m] / \hcf{\Kr}[1]) &\cong \ker \effzero = \ker(\clgp{\Kr}[m] \to \clgp\Kr[1]), \\
\Gal(\hcf{\Kr}[m] / \hcf{\cmfor}[m]) &\cong \ker \effone = \ker(\clgp{\Kr}[m] \to \clgp{K_0^r}[m]), \\
\Gal(\hcf{\Kr}[m] / \CMcmfrcmtr[m]) &\cong \ker \efftwo[m] = \ker(\clgp{\Kr}[m] \to \srcgcmf[m] ).
\end{split}
\end{align}
Therefore,
we can rewrite \eqref{eq:galstar} as
\begin{equation} 
\ker \effzero \supseteq \ker \effone \,\cap\, \ker \efftwo[m]. 
\end{equation}
All groups involved are subgroups of $\clgp\Kr[m]$.
So, we end up with the following algorithm.
\begin{algo} \label{algo:doesstarmhold}
\textsc{Input}. A primitive quartic CM pair $(K, \Phi)$ with reflex $\cmfrcmtr$ and a positive integer $m$.
\\ \textsc{Output}. Returns \textsc{Yes} if \eqref{eq:star} holds for the integer $m$. Otherwise, returns \textsc{No}.
\begin{enumerate}
\item Compute the groups $\clgp\Kr[m], \clgp\Kr[1], \clgp{K_0^r}[m], \srcgcmf[m]$.
\item Compute the kernels of the maps $\effzero, \effone, \efftwo[m]$.
\item Compute the intersection $I = \ker \effone \,\cap\, \ker \efftwo[m]$, a subgroup of $\clgp\Kr[m]$.
\item Compute the intersection $J = \ker \effzero \,\cap\, I$, a subgroup of $\clgp\Kr[m]$.
\item If $I = J$, return \textsc{Yes}. Otherwise, return \textsc{No}.
\end{enumerate}
\end{algo}

\label{fgagmentioned}%
The author has implemented the algorithm that finds $\srcgcmf[\mdls]$ when $\cmf$ is a primitive quartic CM field, and the above \Cref{algo:doesstarmhold} in PARI/GP \cite{pari211}, both of which use algorithms involving morphisms between finitely generated abelian groups, particularly ray class groups. These morphism algorithm implementations are implemented by the author as functions in PARI/GP, collected in a file \texttt{fgag.gp}.

\begin{ex}[continues={ex:runningex},label={ex:runningex2}]
Having explicitly computed the groups $\coker \None$ and $\ker \Ntwo$, we compute
\[
\srcg{\cmf}[2] = 
\genord{[\fraka_1, a_1]}{8}
\times
\genord{[\fraka_2, a_2]}{4},
\]
where
$$
a_1 = -18\alpha_0 - 733
\hspace{2em}
\text{and}
\hspace{2em}
a_2 = -14752\alpha_0 - 600723.
$$

The reflex of the quartic CM pair $\cmfrcmtr$ of $\cmfcmt$ with $\cmfr = \cmfdab{\alphar}[5]{106}{809}$. The ray class field $\Hcmfor[2]$ of its totally real subfield $\cmfor$ is isomorphic to $\cmfor \cong \IQ(\sqrt{5})$.

We do the first step of \Cref{algo:doesstarmhold}. We get
\begin{aligncol*}{2}
\Clcmfr[2] &= \genord{\frakb_1}{16} \times \genord{\frakb_2}{2}, \\
\ker \effzero &= \genord{\frakb_1^8}{2} \times \genord{\frakb_2}{2}, \\
\ker \effonetwo &= \genord{\frakb_1}{16} \times \genord{\frakb_2}{2}, \\
\ker \efftwotwo &= \genord{\frakb_1^8}{2},
\end{aligncol*}
where $\frakb_1 = (11, 1/40\alphar^3 + 73/40\alphar + 1)$ and $\frakb_2 = (-1/20\alphar^3 - 1/40\alphar^2 - 93/20\alphar - 73/40)$. Continuing with the rest of the steps of \Cref{algo:doesstarmhold}, we find that \eqstar{2} holds.
\end{ex}

The following example concerns the formula for $\mS$ given in \Cref{cor:main1}.

\begin{ex}
\writefilename

We take $\cmf = \cmfdab{\alpha}[101]{65}{425}$ and let $\galcl$ be its Galois closure. Fix an embedding $\embtoC: \galcl \to \IC$. Choose the CM type $\Phi$ of $\cmf$ such that both embeddings send $\alpha$ to the positive imaginary axis. Consider the reflex pair $\cmfrcmtr$ of the CM pair $\cmfcmt$. Here, the reflex field is $\cmfr = \cmfdab{\alpha_r}[17]{130}{2525}$.

There are three prime ideals of $\Ocmfr$ over the rational prime $2$, namely
\begin{align*}
\frakp_1 &= 2\Ocmfr + (1/2\alpha_r - 1/2)\Ocmfr, \\
\frakp_2 &= 2\Ocmfr + (1/2\alpha_r + 1/2)\Ocmfr,\text{ and }\\
\frakp_3 &= (1/20{\alpha_r}^2 + 7/4)\Ocmfr. 
\end{align*}
The ideal class group $\Clcmfr$ of $\Kr$ is cyclic
and is generated by the class $[\frakp_1]$.
We may take $S = \{\frakp_1, \frakp_2, \frakp_3\}$.
After verifying that this set $S$ satisfies the hypotheses of \Cref{cor:main1}, we find that $\mS = 8$ and conclude that
$$
H_\Kr(1) \subseteq H_{K_0^r}(8) \CM_{\Kr,\Phi^r}(8).
$$

Computing the ray class group of $\Kr$ for the modulus $8$,
we obtain
$$
\Clcmfr[8]
=
\genord{[\fraka_1]}{48}
\times
\genord{[\fraka_2]}{4}
\times
\genord{[\fraka_3]}{2}
\times
\genord{[\fraka_4]}{2}
\times
\genord{[\fraka_5]}{2}
$$
for certain ideals $\fraka_i$.\footnote{The ideals $\fraka_1, \fraka_2, \fraka_3, \fraka_4, \fraka_5$ are $(443, \frac{1}{2}\alpha_r - \frac{263}{2}), (170999, \frac{1}{2}\alpha_r + \frac{120011}{2}), (41051, \frac{1}{2}\alpha_r - \frac{37351}{2}), (292141, \frac{1}{2}\alpha_r + \frac{198863}{2}), (172229, \frac{1}{2}\alpha_r + \frac{51253}{2})$.}

We compute the kernels of $\effzero, \effone[8], \efftwo[8]$, as defined in \eqref{eq:kereffs}.
We find that
\begin{align*}
\ker \effzero &= 
\genord{[\fraka_1^8\fraka_2^3]}{12}
\times
\genord{[\fraka_1^{24}\fraka_2^2]}{2}
\times
\genord{[\fraka_3]}{2}
\times
\genord{[\fraka_4]}{2}
\times
\genord{[\fraka_5]}{2}, \\
\ker \effone[8] &= 
\genord{[\fraka_1^{25}\fraka_4]}{48}
\times
\genord{[\fraka_2^3\fraka_4\fraka_5]}{4}
\times
\genord{[\fraka_3\fraka_4]}{2},\\
\ker \efftwo[8] &= 
\genord{[\fraka_1^{36}\fraka_2^2\fraka_4\fraka_5]}{4}, \\
\ker \effone[8] \cap \ker \efftwo[8] &=
\genord{[\fraka_1^{24}]}{2}.
\end{align*}


Neither $\ker \effone[8]$ nor $\ker \efftwo[8]$ is contained in $\ker \effzero$. However, their intersection is. Hence, this is an example for the Hilbert class field $H_{\Kr}(1)$ is neither contained in the field $\Kr H_{K_0^r}(8)$ nor the field $\CM_{\Kr,\Phi^r}(8)$ but is contained in the composite of these two fields.

Moreover one can check using \Cref{algo:doesstarmhold} that, in this example, \eqstar{d} does not hold for any proper divisors $d$ of $\mS = 8$. In addition, 
computing the kernels
\begin{align*}
\ker \left(\effone[8]': \Clcmfr[8] \to \Cl_{K_0^r}(4)\right)
&= \Clcmfr[8], \\
\ker \left(\efftwo[8]': \Clcmfr[8] \to \frakC_K(4)\right) &=
\genord{[\fraka_1^{12}]}{4}
\times
\genord{[\fraka_2^{2}]}{2}
\times
\genord{[\fraka_4]}{2}
\times
\genord{[\fraka_5]}{2},
\end{align*}
and the relevant intersections of groups (analogous to what we did in the previous paragraph), we find that neither $H_{K_0^r}(4) \CM_{\Kr,\Phi^r}(8)$ nor $H_{K_0^r}(8) \CM_{\Kr,\Phi^r}(4)$ contains $H_\Kr(1)$.

Finally,
we remark
that $8$ is not the minimum integer $m$
for which \eqref{eq:star} holds.
One can verify that the smallest integer for which \eqref{eq:star} holds is $m = 5$.
This is done by recalling that \eqref{eq:star} does not hold $m = 1, 2, 4$ as they are proper divisors of $8$ and then applying \Cref{algo:doesstarmhold} to $m = 3$ and then $m = 5$.
\end{ex}

\section{Computing \texorpdfstring{$\Hcmfr(1)$}{H_Kr(1)} under \texorpdfstring{(\protect\hyperlink{label:eqstar}{\ensuremath{\star_2}})}{(*2)}}
\label{subsec:algoimpl}
\writefilename

Let $\cmf$ be a primitive quartic CM field and let $\cmt$ be a CM type of $\cmf$.
Let $\cmfrcmtr$ be the reflex pair of $\cmfcmt$.

The field extension $\CMcmfrcmtr[1]$ of $\cmfr$ is obtained by computing Igusa invariants \cite{igusa1967, spallek1994} and appending them to the field $\cmfr$. In a similar manner, the field extension $\CMcmfrcmtr[2]$ is obtained by computing Rosenhain invariants \cite{wamelen1999} instead of Igusa invariants.

\Cref{subsec:theta} gives a brief recall of theta functions, Rosenhain invariants and references how to compute these invariants.

In \Cref{subsec:Hcmfr1intheory}, we discuss how to compute, assuming \eqstar{2}, a defining polynomial for $\Hcmfr[1]$ in terms of these Rosenhain invariants. We also discuss our proof-of-concept implementation in this section. The author is working on how to get similar results when the assumption is \eqstar{m} for integers $m \geq 3$.

In \Cref{subsec:cmbeatskummer}, we compare our implementation of the CM theory algorithm to the current implementations of the well-known Kummer theory algorithm.

  \subsection{Theta constants and Rosenhain invariants}
  \label{subsec:theta}
\writefilename

Let $g \in \IZpos$. Let $\IH_g$ denote the set of $g \times g$ symmetric matrices over $\IC$ with positive definite imaginary part. For any $\va, \vb \in \IQ^g$, the \textit{theta function with characteristic $\colvec{ \va \\ \vb }$} is the function
\begin{align*}
\thfun[\va \\ \vb](\vz, \PM) &= \sum_{\vn \in \IZ^g} \exp\left(\, \pi i \trpos{\vn + \va} \PM (\vn + \va) \right) \exp \left( 2\pi i\, \trpos{\vn + \va} (\vz + \vb)\right)
\end{align*}
 on $\IC^g \times \IH_g$.

Let $a_1, a_2, b_1, b_2 \in \{0, 1/2\}$.
We use the following short-hand notation, used in \cite{dupont2006,streng2010,cosset2011},
for the sixteen theta functions with half-integer characteristics
as follows:
$$
\theta_{16a_2+8a_1+4b_2+2b_1}(\vz, \PM) := \thfun[ \charmat{a_1}{a_2}{b_1}{b_2} ](\vz, \PM).
$$
For each $i \in \{0, \ldots, 15\}$, we denote by $\vartheta_i(\PM)$ the function $\theta_i(\vzero, \PM)$ on $\IH_g$.

Let $\ppas$ be a principally polarized abelian surface.
An ordered basis $\torb = \{B_1, B_2, B_3, B_4\}$ of the $2$-torsion subgroup of $\ppas$ is said to be \textit{symplectic with respect to the Weil pairing} $\mu_2$ if the matrix $[ \mu_2(B_i, B_j ) ]_{i, j}$ is equal to $\Omega_{\IZ/2\IZ}$, where
$\Omega_R$ is the $2g \times 2g$ matrix of the form
$$
\Omega_R \, = \, \begin{bmatrix}\mathbf{0} & \mathbf{I}_g \\ -\mathbf{I}_g & \mathbf{0} \end{bmatrix}
$$
whose entries are in the ring $R$.
The pair $\ppastorb$ is called a \textit{principally polarized abelian surface with level $2$ structure}.
Two principally polarized abelian surfaces with level $2$ structures $\ppastorb$ and $\ppastorbpr$ are said to be isomorphic
if and only if
there exists an isomorphism $a: \ppas \to \ppas'$ of principally polarized abelian surfaces such that $a(B_i) = B_i'$.

We define the following set of functions on $\IH_2$.
\begin{equation} \label{eq:rosinvs}
\lambda_1(\PM) = \left(\frac{\vartheta_0(\PM)\vartheta_1(\PM)}{\vartheta_2(\PM)\vartheta_{3}(\PM)}\right)^2, \,\,
\lambda_2(\PM) = \left(\frac{\vartheta_1(\PM)\vartheta_{12}(\PM)}{\vartheta_{2}(\PM)\vartheta_{15}(\PM)}\right)^2, \,\,
\lambda_3(\PM) = \left(\frac{\vartheta_0(\PM)\vartheta_{12}(\PM)}{\vartheta_{3}(\PM)\vartheta_{15}(\PM)}\right)^2.
\end{equation}
Given $\PM \in \IH_2$, we say that $\blambda(\PM) := (\lambda_1(\PM), \lambda_2(\PM), \lambda_3(\PM))$ is \textit{a triple of Rosenhain invariants for $\PM$}. There are other possible choices for defining the Rosenhain invariants but we fix the above choice in this article for the sake of simplicity. This choice is the same as Gaudry's 2007 article\cite{gaudry2007}. The reader interested in learning more about Rosenhain invariants is referred to \cite{mumford1983,mumford2012}.

Let $\PM \in \IH_2$ and consider the complex hyperelliptic curve $\hypell_\PM$:
$$
\hypell_\PM: y^2 = x(x - 1)(x - \lambda_1(\PM))(x - \lambda_2(\PM))(x - \lambda_3(\PM)).
$$
We denote its hyperelliptic involution map $(x, y) \mapsto (x, -y)$ by $\hypinv$.
For each $i \in \{0, 1\}$, we denote by $P_i$ the point $(i, 0) \in \hypell$
and for each $j \in \{1, 2, 3\}$, we denote by $Q_j$ the point $(\lambda_j(\PM), 0)$.
The points fixed by $\hypinv$, called the Weierstrass points of $\hypell_\PM$, are then $P_0, P_1, Q_1, Q_2, Q_3$ and the point $\infty$ at infinity.

The Jacobian $\Jac[\hypell]$ of $\hypell_\PM$ is a complex principally polarized abelian surface $\ppas_\PM$ isomorphic to the torus $\IC^2 / (\IZ^2 + \PM \IZ^2)$. The Jacobian $\Jac[\hypell]$ is isomorphic to the quotient of the group of degree-zero divisors of $\hypell$ by its subgroup of principal divisors, and so elements of a Jacobian can be thought of as divisor classes. 
A basis $\torb_\PM$ for the $2$ torsion subgroup $\ppas_\PM[2]$ of $\ppas_\PM$ is given by $\torb_\PM = \{B_1, B_2, B_3, B_4\}$ where
\begin{equation} \label{eq:torbasis}
\begin{split}
B_1 &= [Q_1 + Q_2 - 2 \infty] \\
B_2 &= [P_0 + P_1 - 2 \infty] \\
B_3 &= [Q_2 + Q_3 - 2 \infty] \\
B_4 &= [P_0 - \infty].
\end{split}
\end{equation}
One can verify that not only is this a basis for $\ppas_\PM[2]$, but it is also symplectic with respect to the Weil pairing $\mu_2$. Just like the choice of Rosenhain invariants, one could use other choices to associate a different symplectic basis to the hyperelliptic curve $\hypell_\PM$ but we use this choice for the rest of the article. Hence the pair $\ppastorb[\PM]$ is a principally polarized abelian surface with level $2$ structure.

\begin{ex}[continues={ex:runningex2},label={ex:runningex3}]
We continue with the CM field $$\cmf = \cmfdab{\alpha}[809]{53}{500}.$$
This CM field has two pairs of CM types up to equivalence.
They are:
\begin{aligncol*}{2}
\Phi &= \{ \, \alpha \mapsto 6.3813\ldots i \, \, , \, \,  \alpha \mapsto 3.5041\ldots i \, \} \\
\Phi' &= \{ \, \alpha \mapsto 6.3813\ldots i \, \, , \, \,  \alpha \mapsto -3.5041\ldots i \, \}
\end{aligncol*}
We consider a complex principally polarized abelian surface $\ppas$ with complex multiplication by $\Ocmf$ isomorphic to $\IC^2 / (\Phi(\frako))$ where $\frako$ is the ideal $(49, \alpha+5)$ of $\Ocmf$. Our principally polarized abelian surface $\ppas$ is isomorphic to $\ppas_\PM$
where
\[
\PM \approx \begin{bmatrix}
1.5852i & -1.6036 \\
-1.6036 & 1/2 + 1.7723i
\end{bmatrix}.
\] 
Defining $\torb_\PM$ as in \eqref{eq:torbasis}, we find that $\ppastorb[\PM]$ is a principally polarized abelian surface with level $2$ structure over $\IC$. 
\end{ex}

If $\ppas_\PM$ has complex multiplication by $\Ocmf$,
that is $\End A \cong \Ocmf$,
then the set $\blambda(\PM)$ consists only of algebraic numbers.
The field of moduli of $\ppastorb[\PM]$ is $\IQ(\blambda(\PM))$.
This fact, combined with \cite[Corollary 18.9]{shimura1998}
gives the following equality of fields:
\begin{equation} \label{eq:cmkr}
\CMcmfrcmtr[2] = \Kr( \, \blambda(\PM) \, ).
\end{equation}

  \subsection{The algorithm and implementation}
  \label{subsec:Hcmfr1intheory}
\writefilename

Let $\cmfcmt$ be a primitive quartic CM pair and let $\cmfrcmtr$ be its reflex pair.

We discuss how to obtain $\Hcmfr[1]$ assuming \eqstar{2} holds. Recall that we may use \Cref{algo:doesstarmhold} to determine whether or not \eqstar{2} holds, which in turn lets us know if we can use this method or not. Under Stark's conjectures, for any positive integer $m$, the field $\Hcmfor[m]$ can be obtained complex-analytically using Stark units \cite{roblot2000}. 
In practice, this works. If we encounter a case where it fails, then this case is a contradiction to Stark's conjectures and we would fallback to Kummer-theory based algorithms.
Once we have computed a set of defining polynomials for the compositum, we may then use Galois theory via the technique described in \cite[Section 6]{engemorain2003} to determine a set of defining polynomials for to obtain $\Hcmfr[1]$. 

\writefilename

The author has written a partial implementation of the above algorithm. The implementation is meant to show, as a proof-of-concept, that CM theory algorithms can compute Hilbert class fields of various examples of quartic CM fields which the current implementations of the known algorithms cannot. We elaborate more on this in \Cref{subsec:cmbeatskummer}. 
This implementation uses SageMath to interface with the PARI/GP \texttt{fgag.gp} script (see page~\pageref{fgagmentioned}) in order to find $\srcgcmf[\frakm]$ and the group $\Idcmfrcmtr[m]$. In order to compute period matrices of the relevant principally polarized abelian surfaces and to compute the action of $\Clcmfr[m]$ on the Rosenhain invariants using explicit Shimura reciprocity \cite{streng2012}, we use Streng's SageMath \textsc{Recip} package. Finally, SageMath
interfaces with PARI/GP to approximate the needed theta constants, using the algorithm in \cite{engethome2013}, for the Rosenhain invariant computations.

\begin{ex}[continues={ex:runningex3},label={ex:runningex4}]
Take $$G = \ker \effonetwo = \ker(\Clcmfr[2] \to \srcgcmf[2])$$ and $$H = \ker \effzerotwo = \ker(\Clcmfr[2] \to \Clcmfr[1]).$$ By applying explicit Shimura reciprocity, we find, for each $[\frakg] \in G$, a set of integers $i_0, i_1, i_2, i_3$, a root of unity $\mu$, and a period matrix $\PM'$ such that
$$
\lambda_1(\PM)^{[\frakg]} = \left(\frac{\vartheta_{i_0}(\PM')\vartheta_{i_1}(\PM')}{\vartheta_{i_2}(\PM')\vartheta_{i_3}(\PM)}\right)^2.
$$
For each period matrix encountered in the previous step, we may then compute approximations of the squares of the relevant theta constants.

Consider the polynomial
$$
p_1(X) = \prod_{[\fraka] \in G/H} \left( X - \sum_{h \in H} \widetilde\lambda_1(\PM)^{[\frakh][\fraka]}\right) 
$$
where for each $[\frakg] \in G$, the value $\widetilde\lambda_1(\PM)^{[\frakg]}$ is an approximation of $\lambda_1(\PM)^{[\frakg]}$ with a sufficiently high precision for the next steps. Using the approximating polynomial $\widetilde{p}_1(X)$ and the methods in \cite[Section II.10]{streng2010} to recover a polynomial with coefficients in $\cmfr$ from this approximation, we find a defining polynomial $p_1(X)$ of the extension $\cmfr(\lambda_1(\PM))/\cmfr$ as follows:
\begin{dmath*}
d \cdot p_1(X) = d X^8 + (-301220403431369353045700111055149125{\alpha_r}^2
\\
- 29438063764199719491190907374578941125) X^7 + \ldots
\end{dmath*}
with $d = 5^5 \cdot 11^4 \cdot 71^2 \cdot 251^4 \cdot 311^2 \cdot 431^2$. Observe that this polynomial satisfies $$p_1(X)~\in~\frac{1}{d}~\Ocmfr[X]~\subseteq~\cmfr[X].$$ Finally, one may verify that the extension of $\cmfr$ defined by this polynomial is unramified and cyclic of
degree $8$. This means that $$\Hcmfr(1) \cong \cmfr[X]/(p_1(X)).$$%
\end{ex}

  \subsection{Comparison to the existing Kummer theory algorithm}
  \label{subsec:cmbeatskummer}
\writefilename

Other methods are known for computing abelian extensions, and in particular Hilbert class fields.

One approach is an algorithm based on Kummer theory \cite{fieker2001}.
This algorithm can find abelian extensions $\nfext$, say of degree $d$,
of a general number field $\nf$.
However, this algorithm requires that the base field $\nf$
has sufficiently many roots of unity or that we
consider the larger field $\nf' = \nf(\zeta_d)$ to find abelian extensions of the original base field $\nf$.
Such an algorithm is at a disadvantage when $\nf$ is a quartic CM field
and the required extension $L$ has a large prime power degree,
since working with a larger field $\nf' = \nf(\zeta_d)$ is required to use the algorithm.

Using current implementations of the Kummer algorithm on fields whose ideal class groups are cyclic with order a power of $2$, we are able to find defining polynomials for the Hilbert class fields of primitive quartic CM fields whose class groups are cyclic of order up to $16$. 

For primitive quartic CM fields whose class groups are cyclic of order $32$, there are examples for which our implementation of the CM theory algorithm outdoes current implementations of the Kummer theory algorithm. For example, computing the cyclic degree $32$ extension
$\Hcmfr$ of $$\cmfr = \cmfdab{\alpha}[53]{104}{796},$$
which satisfies \eqstar{2}, 
takes less than 10 minutes using our implementation
while the \texttt{bnrclassfield} function of PARI
and the \texttt{HilbertClassField} function of MAGMA do not finish within twenty-four hours \footnote{The machine used is a laptop with 16 GB of RAM. Its processor was an \textit{AMD Ryzen 7 2700U}.}.

\textbf{Acknowledgments.} I would like to thank my supervisors, Andreas Enge and Marco Streng, for the careful reading and the invaluable suggestions and comments they gave.




%
\printbibliography

\end{document}